\documentclass{amsproc}
\usepackage{amssymb}
\usepackage{graphicx}
\usepackage{euscript}
\usepackage{mathrsfs} 
\usepackage{units}   

\makeatletter
\@namedef{subjclassname@2010}{%
\textup{2010} Mathematics Subject Classification}
\makeatother

\numberwithin{equation}{subsection}

\newtheorem{thm}{Theorem}[subsection]
\newtheorem{cor}[thm]{Corollary}
\newtheorem{lem}[thm]{Lemma}
\newtheorem{pro}[thm]{Proposition}

\newtheorem*{thm*}{Theorem}

\theoremstyle{remark}
\newtheorem{rem}[thm]{Remark}

\theoremstyle{definition}

\DeclareMathOperator{\lin}{\mbox{\sc lin}}
\DeclareMathOperator{\D}{d}
\DeclareMathOperator{\dess}{{\mathsf{Des}}}
\DeclareMathOperator{\dzii}{{\mathsf{Chi}}}

\DeclareMathOperator{\koo}{{\mathsf{root}}}
\DeclareMathOperator{\Koo}{{\mathsf{Root}}}
\DeclareMathOperator{\paa}{{\mathsf{par}}}

\newcommand*{\borel}[1]{{\mathfrak B}(#1)}
\newcommand*{\cbb}{\mathbb C}
\newcommand*{\cc}{\mathcal C}
\newcommand*{\card}[1]{\mathrm{card}(#1)}
\newcommand*{\des}[1]{{\dess(#1)}}
\newcommand*{\dz}[1]{{\EuScript D}(#1)}
\newcommand*{\dzi}[1]{\dzii(#1)}
\newcommand*{\dzin}[2]{\dzii^{\langle#1\rangle}(#2)}
\newcommand*{\dzn}[1]{{\EuScript D}^\infty(#1)}
\newcommand*{\ee}{\mathcal E}
\newcommand*{\escr}{{\mathscr{E}_V}}
\newcommand*{\ff}{\mathcal F}
\newcommand*{\Ge}{\geqslant}
\newcommand*{\hh}{\mathcal H}

\newcommand*{\is}[2]{\langle#1,#2\rangle}

\newcommand*{\kk}{\mathcal K}
\newcommand*{\Ko}[1]{\Koo(#1)}
\newcommand*{\lambdab}{{\boldsymbol\lambda}}
\newcommand*{\lambdabxx}[2]{\lambdab^{\hspace{-.3ex} \langle #1 \rangle}_{u|#2}}
\newcommand*{\lambdai}[1]{{\lambda_{#1}^{\hspace{-.3ex}{\langle} i  \rangle}}}
\newcommand*{\lambdabi}{{\lambdab^{\hspace{-.3ex}\langle i \rangle}}}
\newcommand*{\Le}{\leqslant}
\newcommand*{\mm}{\mathscr M}
\newcommand*{\mui}[1]{{\mu_{#1}^{\hspace{-.25ex}\langle i \rangle}}}
\newcommand*{\nbb}{\mathbb N}
\newcommand*{\ogr}[1]{\boldsymbol B(#1)}

\newcommand*{\pa}[1]{\paa(#1)}
\newcommand*{\quasi}[1]{\mathscr Q(#1)}
\newcommand*{\rbb}{\mathbb R}
\newcommand*{\slam}{S_{\boldsymbol \lambda}}
\newcommand*{\smalloplus}{\raise0pt\hbox{$\scriptscriptstyle \oplus$}}

\newcommand*{\sti}[1]{\mathscr S(#1)}
\newcommand*{\supp}[1]{\mathrm{supp}\,#1}
\newcommand*{\tcal}{{\mathscr T}}

\newcommand*{\varepsiloni}[1]{{\varepsilon_{#1}^{\hspace{-.1ex}\langle i \rangle}}}
\newcommand*{\xx}{\mathcal X}
\newcommand*{\zbb}{\mathbb Z}

\begin{document}
   \title[Unbounded subnormal weighted shifts on directed trees]
{Unbounded subnormal weighted shifts on directed
trees}
   \author[P.\ Budzy\'{n}ski]{Piotr Budzy\'{n}ski}
   \address{Katedra Zastosowa\'{n} Matematyki,
Uniwersytet Rolniczy w Krakowie, ul.\ Balicka 253c,
PL-30198 Krak\'ow}
   \email{piotr.budzynski@ur.krakow.pl}
   \author[Z.\ J.\ Jab{\l}o\'nski]{Zenon Jan Jab{\l}o\'nski}
\address{Instytut Matematyki, Uniwersytet Jagiello\'nski,
ul.\ \L ojasiewicza 6, PL-30348 Kra\-k\'ow, Poland}
   \email{Zenon.Jablonski@im.uj.edu.pl}
   \author[I.\ B.\ Jung]{Il Bong Jung}
   \address{Department of Mathematics, Kyungpook
National University, Daegu 702-701, Korea}
   \email{ibjung@knu.ac.kr}
   \author[J.\ Stochel]{Jan Stochel}
\address{Instytut Matematyki, Uniwersytet Jagiello\'nski,
ul.\ \L ojasiewicza 6, PL-30348 Kra\-k\'ow, Poland}
   \email{Jan.Stochel@im.uj.edu.pl}
   \thanks{Research of the first, second
and fourth authors was supported by the MNiSzW
(Ministry of Science and Higher Education) grant NN201
546438 (2010-2013). The third author was supported by
Basic Science Research Program through the National
Research Foundation of Korea (NRF) grant funded by the
Korea government (MEST) (2009-0093125).}
   \begin{abstract}
A new method of verifying the subnormality of
unbounded Hilbert space operators based on an
approximation technique is proposed. Diverse
sufficient conditions for subnormality of unbounded
weighted shifts on directed trees are established. An
approach to this issue via consistent systems of
probability measures is invented. The role played by
determinate Stieltjes moment sequences is elucidated.
Lambert's characterization of subnormality of bounded
operators is shown to be valid for unbounded weighted
shifts on directed trees that have sufficiently many
quasi-analytic vectors, which is a new phenomenon in
this area. The cases of classical weighted shifts and
weighted shifts on leafless directed trees with one
branching vertex are studied.
   \end{abstract}
   \maketitle
   \section{Introduction}
The theory of bounded subnormal operators was
originated by P. Halmos in \cite{hal1}. Nowadays, its
foundations are well-developed (see \cite{con2}; see
also \cite{c-f} for a recent survey article on this
subject). The theory of unbounded symmetric operators
had been established much earlier (see \cite{jvn} and
the monograph \cite{stone}). In view of Naimark's
theorem, these particular operators resemble {\em
unbounded} subnormal operators, i.e., operators having
normal extensions in (possibly larger) Hilbert spaces.
The first general results on unbounded subnormal
operators appeared in \cite{bis} and \cite{foi} (see
also \cite{sli}). A systematic study of this class of
operators was undertaken in the trilogy
\cite{StSz3,StSz1,StSz4}. The theory of unbounded
subnormal operators has intimate connections with
other branches of mathematics and quantum physics (see
\cite{sz2,at-cha1,at-cha2} and
\cite{jor,stob,sz1,j-s}). It has been developed in two
main directions, the first is purely theoretical (cf.\
\cite{m-s,jin,StSz2,e-v,vas1,dem1,dem2,dem3,vas2,vas3,al-vas}),
the other is related to special classes of operators
(cf.\ \cite{c-j-k,kou,k-t1,k-t2}). In this paper, we
will focus our attention mostly on the class of
weighted shifts on directed trees.

The notion of a weighted shift on a directed tree
generalizes that of a weighted shift on the $\ell^2$
space, the classical object of operator theory (see
e.g., the monograph \cite{nik} on the unilateral shift
operator, \cite{shi} for a survey article on bounded
unilateral and bilateral weighted shifts, and
\cite{ml} for basic facts on unbounded ones). In a
recent paper \cite{j-j-s}, we have studied some
fundamental properties of weighted shifts on directed
trees. Although considerable progress has been made in
this field, a number of fundamental questions have not
been answered. Our aim in this paper is to continue
investigations along these lines with special emphasis
put on the issue of subnormality of unbounded
operators, the case which is essentially more
complicated and not an easy extension of the bounded
one. The main difficulty comes from the fact that the
celebrated Lambert characterization of subnormality of
bounded operators (cf.\ \cite{Lam}) is no longer valid
for unbounded ones (see Section \ref{subs1}; see also
\cite{j-j-s4} for a surprising counterexample). A new
criterion (read:\ sufficient condition) for
subnormality of unbounded operators has been invented
recently in \cite{c-s-sz}. By using it, we will show
that subnormality is preserved by the operation of
taking a certain limit (see Theorem \ref{tw1}). This
enables us to perform the approximation procedure
relevant to unbounded weighted shifts on directed
trees. What we get is Theorem \ref{main}, which is the
main result of this paper. It provides a criterion for
subnormality of unbounded weighted shifts on directed
trees written in terms of consistent systems of
measures (which is new even in the case of bounded
operators). Roughly speaking, for bounded and some
unbounded operators having dense set of
$C^\infty$-vectors, the assumption that
$C^\infty$-vectors generates Stieltjes moment
sequences implies subnormality. As discussed in
Section \ref{subs1}, there are unbounded operators for
which this is not true (the reverse implication is
always true, cf.\ Proposition \ref{necess-gen}). It is
a surprising fact that there are non-hyponormal
operators having dense set of $C^\infty$-vectors
generating Stieltjes moment sequences. These are
carefully constructed weighted shifts on a leafless
directed tree with one branching vertex (cf.\
\cite{j-j-s4}). The same operators do not satisfy the
consistency condition $2^\circ$ of Lemma
\ref{charsub2} and none of them has consistent system
of measures.

Under some additional assumption, the criterion for
subnormality formulated in Theorem \ref{main} becomes
a full characterization (cf.\ Corollary
\ref{necessdet2}). This is the case in the presence of
quasi-analytic vectors (cf.\ Theorem \ref{main-0}),
which is the first result of this kind (see Section
\ref{cfs} for more comments).

It is worth mentioning that our method of proving
Theorem \ref{main} depends essentially on the passage
through weighted shifts that may have zero weights.

The assumption that all basic vectors coming from
vertices of the directed tree are $C^\infty$-vectors
diminishes the class of weighted shifts to which
Theorem \ref{main} can be applied. Note that there are
weighted shifts on directed trees with nonzero
weights, whose squares have trivial domain (directed
trees admitting such pathological weighted shifts are
the largest possible, cf.\ \cite{j-j-s3}).
Unfortunately, the known criteria for subnormality
that can be applied to such operators seems to be
useless (see Section \ref{cfsub} for more comments).

It was shown in \cite{j-j-s2} that, in most cases, a
normal extension of a nonzero subnormal weighted shift
on a directed tree $\tcal$ with nonzero weights could
not be modelled as a weighted shift on a directed tree
$\hat \tcal$ (no relationship between $\tcal$ and
$\hat\tcal$ is required); the only exceptional cases
are those in which the directed tree $\tcal$ is
isomorphic either to $\zbb$ or to $\zbb_+$.

Though our Theorem \ref{main} provides only sufficient
conditions for subnormality of weighted shifts on
directed trees, in the case of classical weighted
shifts it gives the full characterization (cf.\
Section \ref{cws}). The case of leafless directed
trees with one branching vertex is discussed in
Section \ref{obv} (see \cite{j-j-s4} for new phenomena
that happen for weighted shifts on such simple
directed trees).
   \section{Preliminaries}
   \subsection{Notation and terminology}
Let $\zbb$, $\rbb$ and $\cbb$ stand for the sets of
integers, real numbers and complex numbers
respectively. Define
   \begin{align*}
\text{$\zbb_+ = \{0,1,2,3,\ldots\}$, $\nbb =
\{1,2,3,4,\ldots\}$ and $\rbb_+ = \{x \in \rbb \colon x \Ge
0\}$.}
   \end{align*}
We write $\borel{\rbb_+}$ for the $\sigma$-algebra of all
Borel subsets of $\rbb_+$. The closed support of a positive
Borel measure $\mu$ on $\rbb_+$ is denoted by $\supp \mu$.
We write $\delta_0$ for the Borel probability measure on
$\rbb_+$ concentrated at $0$. We denote by $\card Y$ the
cardinal number of a set $Y$.

Let $A$ be an operator in a complex Hilbert space
$\hh$ (all operators considered in this paper are
linear). Denote by $\dz{A}$ and $A^*$ the domain and
the adjoint of $A$ (in case it exists). Set $\dzn{A} =
\bigcap_{n=0}^\infty\dz{A^n}$; members of $\dzn{A}$
are called {\em $C^\infty$-vectors} of $A$. A linear
subspace $\ee$ of $\dz{A}$ is said to be a {\em core}
of $A$ if the graph of $A$ is contained in the closure
of the graph of the restriction $A|_{\ee}$ of $A$ to
$\ee$. If $A$ is closed, then $\ee$ is a core of $A$
if and only if $A$ coincides with the closure of
$A|_{\ee}$. A closed densely defined operator $N$ in
$\hh$ is said to be {\em normal} if $N^*N=NN^*$
(equivalently:\ $\dz{N}=\dz{N^*}$ and
$\|N^*h\|=\|Nh\|$ for all $h \in \dz{N}$). For other
facts concerning unbounded operators (including normal
ones) that are needed in this paper we refer the
reader to \cite{b-s,weid}. A densely defined operator
$S$ in $\hh$ is said to be {\em subnormal} if there
exists a complex Hilbert space $\kk$ and a normal
operator $N$ in $\kk$ such that $\hh \subseteq \kk$
(isometric embedding) and $Sh = Nh$ for all $h \in \dz
S$. It is clear that subnormal operators are closable
and their closures are subnormal.

In what follows, $\ogr \hh$ stands for the
$C^*$-algebra of all bounded operators $A$ in $\hh$
such that $\dz{A}=\hh$. We write $\lin \ff$ for the
linear span of a subset $\ff$ of $\hh$.
   \subsection{Directed trees}
   Let $\tcal=(V,E)$ be a directed graph (i.e., $V$ is
the set of all vertices of $\tcal$ and $E$ is the set
of all edges of $\tcal$). If for a given vertex $u \in
V$, there exists a unique vertex $v\in V$ such that
$(v,u)\in E$, then we say that $u$ has a parent $v$
and write $\pa u$ for $v$. Since the correspondence $u
\mapsto \pa u$ is a partial function (read:\ a
relation) in $V$, we can compose it with itself
$k$-times ($k \in \nbb$); the result is denoted by
$\paa^k$ ($\paa^0$ is the identity mapping on $V$). A
vertex $v$ of $\tcal$ is called a {\em root} of
$\tcal$, or briefly $v \in \Ko \tcal$, if there is no
vertex $u$ of $\tcal$ such that $(u,v)$ is an edge of
$\tcal$. Note that if $\tcal$ is connected and each
vertex $v \in V^\circ:=V\setminus \Ko{\tcal}$ has a
parent, then the set $\Ko{\tcal}$ has at most one
element (cf.\ \cite[Proposition 2.1.1]{j-j-s}). If
$\Ko \tcal$ is a one-point set, then its unique
element is denoted by $\koo$. We say that a directed
graph $\tcal$ is a {\em directed tree} if $\tcal$ is
connected, has no circuits and each vertex $v \in
V^\circ$ has a parent $\pa v$.

Let $\tcal=(V,E)$ be a directed tree. Set $\dzi u =
\{v\in V\colon (u,v)\in E\}$ for $u \in V$. A member
of $\dzi u$ is called a {\em child} (or {\em
successor}) of $u$. We say that $\tcal$ is {\em
leafless} if $V = V^\prime$, where $V^\prime:=\{u \in
V \colon \dzi u \neq \varnothing\}$. It is clear that
every leafless directed tree is infinite. A vertex $u
\in V$ is called a {\em branching vertex} of $\tcal$
if $\card{\dzi{u}} \Ge 2$.

It is well-known that (see e.g., \cite[Proposition
2.1.2]{j-j-s}) if $\tcal$ is a directed tree, then
$\dzi u \cap \dzi v = \varnothing$ for all $u, v\in V$
such that $u \neq v$, and
   \begin{align} \label{roz}
V^\circ= \bigsqcup_{u\in V} \dzi u.
   \end{align}
(The symbol ``\,$\bigsqcup$\,'' denotes disjoint
union of sets.) For a subset $W \subseteq V$, we
put $\dzi W = \bigsqcup_{v \in W} \dzi v$ and
define $\dzin{0}{W} = W$, $\dzin{n+1}{W} =
\dzi{\dzin{n}{W}}$ for $n\in \zbb_+$ and $\des W
= \bigcup_{n=0}^\infty \dzin n W$. By induction,
we have
   \allowdisplaybreaks
   \begin{align} \label{n+1}
\dzin{n+1}{W} & = \bigcup_{v \in \dzi{W}}
\dzin{n}{\{v\}}, \quad n \in \zbb_+,
   \\
\label{chmn} \dzin{m}{\dzin{n}{W}} & =
\dzin{m+n}{W}, \quad m,n \in \zbb_+.
   \end{align}
We shall abbreviate $\dzin n {\{u\}}$ and
$\des{\{u\}}$ to $\dzin n u$ and $\des{u}$
respectively. We now state some useful properties of
the functions $\dzin{n}{\cdot}$ and $\des{\cdot}$.
   \begin{pro}
If $\tcal$ is a directed tree, then
   \begin{align}  \label{num4}
\dzin{n}{u} &= \{w \in V\colon \paa^n(w)=u\}, \quad n
\in \zbb_+,\, u \in V,
   \\
\label{dzinn2} \dzin{n+1}{u} & = \bigsqcup_{v \in
\dzi{u}} \dzin{n}{v}, \quad n \in \zbb_+,\, u \in V,
   \\
\label{num1} \dzin{n+1}{u} & = \bigsqcup_{v \in
\dzin{n}{u}} \dzi{v}, \quad n \in \zbb_+,\, u \in V,
   \\
\des u & = \bigsqcup_{n=0}^\infty \dzin n u, \quad
u\in V, \label{num3}
   \\
\des{u_1} \cap \des{u_2} & = \varnothing, \quad u_1,
u_2 \in \dzi{u},\, u_1 \neq u_2,\, u \in V.
\label{num3+}
   \end{align}
   \end{pro}
   \begin{proof}
Equality \eqref{num4} follows by induction on $n$.
Combining \eqref{n+1} with the fact that the sets
$\dzin{n} u$, $u \in V$, are pairwise disjoint for
every fixed integer $n \Ge 0$, we get \eqref{dzinn2}.
Equality \eqref{num1} follows from the definition of
$\dzin{n+1}{u}$ and \eqref{roz}. Using the definition
of $\paa$ and the fact that $\tcal$ has no circuits, we
deduce that the sets $\dzin n u$, $n \in \zbb_+$, are
pairwise disjoint. Hence, \eqref{num3} holds. Assertion
\eqref{num3+} can be deduced from \eqref{num4} and
\eqref{num3}.
   \end{proof}
    \begin{pro}[\mbox{\cite[Corollary 2.1.5]{j-j-s}}]
\label{przem} If $\tcal$ is a directed tree with root,
then $V = \des{\koo} = \bigsqcup_{n=0}^\infty \dzin n
{\koo}$.
   \end{pro}
   \subsection{Weighted shifts on directed trees}
In what follows, given a directed tree $\tcal$, we
tacitly assume that $V$ and $E$ stand for the sets of
vertices and edges of $\tcal$ respectively. Denote by
$\ell^2(V)$ the Hilbert space of all square summable
complex functions on $V$ with the standard inner
product $\is fg = \sum_{u \in V} f(u)
\overline{g(u)}$. For $u \in V$, we define $e_u \in
\ell^2(V)$ to be the characteristic function of the
one-point set $\{u\}$. Then $\{e_u\}_{u\in V}$ is an
orthonormal basis of $\ell^2(V)$. Set $\escr = \lin
\{e_u\colon u \in V\}$.

Given $\lambdab = \{\lambda_v\}_{v \in V^\circ}
\subseteq \cbb$, we define the operator $\slam$ in
$\ell^2(V)$ by
   \begin{align*}
   \begin{aligned}
\dz {\slam} & = \{f \in \ell^2(V) \colon
\varLambda_\tcal f \in \ell^2(V)\},
   \\
\slam f & = \varLambda_\tcal f, \quad f \in \dz
{\slam},
   \end{aligned}
   \end{align*}
where $\varLambda_\tcal$ is the mapping defined on
functions $f\colon V \to \cbb$ via
   \begin{align} \label{lamtauf}
(\varLambda_\tcal f) (v) =
   \begin{cases}
\lambda_v \cdot f\big(\pa v\big) & \text{ if } v\in
V^\circ,
   \\
0 & \text{ if } v=\koo.
   \end{cases}
   \end{align}
We call $\slam$ a {\em weighted shift} on the
directed tree $\tcal$ with weights
$\lambdab=\{\lambda_v\}_{v \in V^\circ}$.

Now we select some properties of weighted shifts on
directed trees that will be needed in this paper (see
Propositions 3.1.2, 3.1.3, 3.1.8, 3.4.1, 3.1.7 and
3.1.10 in \cite{j-j-s}). In what follows, we adopt the
convention that $\sum_{v\in\varnothing} x_v=0$.
   \begin{pro}\label{bas}
Let $\slam$ be a weighted shift on a directed tree
$\tcal$ with weights $\lambdab = \{\lambda_v\}_{v \in
V^\circ}$. Then the following assertions hold{\em :}
   \begin{enumerate}
   \item[(i)] $\slam$ is closed,
   \item[(ii)]
$e_u \in \dz{\slam}$ if and only if $\sum_{v\in\dzi u}
|\lambda_v|^2 < \infty${\em ;} if $e_u \in
\dz{\slam}$, then
   \begin{align} \label{eu}
\slam e_u = \sum_{v\in\dzi u} \lambda_v e_v \quad
\text{and} \quad \|\slam e_u\|^2 = \sum_{v\in\dzi u}
|\lambda_v|^2,
   \end{align}
   \item[(iii)]
$\overline{\dz{\slam}}=\ell^2(V)$ if and only if
$\escr \subseteq \dz{\slam}$,
   \item[(iv)]  if
$\overline{\dz{\slam}}=\ell^2(V)$, then $\escr$ is a
core of $\slam$,
   \item[(v)] $\slam \in \ogr{\ell^2(V)}$ if and
only if $\alpha_{\lambdab}:=\sup_{u\in
V}\sum\nolimits_{v\in\dzi u} |\lambda_v|^2 <
\infty${\em ;} if $\slam \in \ogr{\ell^2(V)}$, then
$\|\slam\|^2=\alpha_{\lambdab}$,
   \item[(vi)]
if $\overline{\dz{\slam}}=\ell^2(V)$, then $\escr
\subseteq \dz{\slam^*}$ and
   \begin{align} \label{sl*}
\slam^*e_u=
   \begin{cases}
   \overline{\lambda_u} e_{\pa u} & \text{if } u \in V^\circ, \\
0 & \text{if } u = \koo,
   \end{cases}
   \quad u \in V,
   \end{align}
   \item[(vii)]
$\slam$ is injective if and only if $\tcal$ is
leafless and $\sum_{v\in\dzi u} |\lambda_v|^2 > 0$ for
every $u\in V$,
   \item[(viii)] if $\overline{\dz{\slam}}=\ell^2(V)$
and $\lambda_v \neq 0$ for all $v \in V^\circ$, then
$V$ is at most countable.
   \end{enumerate}
   \end{pro}
   \subsection{Backward extensions of Stieltjes moment sequences}
We say that a sequence $\{t_n\}_{n=0}^\infty$ of real
numbers is a {\em Stieltjes moment sequence} if there
exists a positive Borel measure $\mu$ on $\rbb_+$ such
that
   \begin{align*}
t_{n}=\int_0^\infty s^n \D\mu(s),\quad n\in \zbb_+,
   \end{align*}
where $\int_0^\infty$ means integration over the set
$\rbb_+$; $\mu$ is called a {\em representing measure}
of $\{t_n\}_{n=0}^\infty$. A Stieltjes moment sequence
is said to be {\em determinate} if it has only one
representing measure. By the Stieltjes theorem (cf.\
\cite[Theorem~ 1.3]{sh-tam} or \cite[Theorem
6.2.5]{ber}), a sequence $\{t_n\}_{n=0}^\infty
\subseteq \rbb$ is a Stieltjes moment sequence if and
only if the sequences $\{t_n\}_{n=0}^\infty$ and
$\{t_{n+1}\}_{n=0}^\infty$ are positive definite
(recall that a sequence $\{t_n\}_{n=0}^\infty
\subseteq \rbb$ is said to be {\em positive definite}
if $\sum_{k,l=0}^n t_{k+l} \alpha_k
\overline{\alpha_l} \Ge 0$ for all $\alpha_0,\ldots,
\alpha_n \in \cbb$ and $n \in \zbb_+$). It is clear
from the definition that
   \begin{align}  \label{st+1}
\text{if $\{t_n\}_{n=0}^\infty$ is a Stieltjes moment
sequence, then so is $\{t_{n+1}\}_{n=0}^\infty$.}
   \end{align}
The converse is not true in general. For example, the
sequence of the form $\{t_{n}\}_{n=0}^\infty=\{t_0,1,
0, 0, \ldots\}$ is never a Stieltjes moment sequence,
but $\{t_{n+1}\}_{n=0}^\infty = \{1, 0, 0, \ldots\}$
is (see Lemma \ref{bext} below for more detailed
discussion of this issue). Moreover, if
$\{t_n\}_{n=0}^\infty$ is an indeterminate Stieltjes
moment sequence, then so is $\{t_{n+1}\}_{n=0}^\infty$
(see Lemma \ref{bext}; see also \cite[Proposition
5.12]{sim}). The converse implication fails to hold
(cf.\ \cite[Corollary 4.21]{sim}; see also
\cite{j-j-s4}).

The question of backward extendibility of Hamburger
moment sequences has well-known solutions (see e.g.,
\cite{wri} and \cite{sz}). Below, we formulate a
solution of a variant of this question for Stieltjes
moment sequences (see \cite[Lemma 6.1.2]{j-j-s} for
the special case of compactly supported representing
measures; see also \cite[Proposition 8]{cur} for a
related matter).
   \begin{lem} \label{bext}
Let $\{t_n\}_{n=0}^\infty$ be a Stieltjes moment
sequence and let $\vartheta$ be a positive real
number. Set $t_{-1}=\vartheta$. Then the following are
equivalent{\em :}
   \begin{enumerate}
   \item[(i)] $\{t_{n-1}\}_{n=0}^\infty$ is a
Stieltjes moment sequence,
   \item[(ii)]  $\{t_{n-1}\}_{n=0}^\infty$ is positive
definite,
   \item[(iii)] there is a representing measure $\mu$
of $\{t_n\}_{n=0}^\infty$ such
that\/\footnote{\;\label{foot}We adhere to the
convention that $\frac 1 0 := \infty$. Hence,
$\int_0^\infty \frac 1 s \D \mu(s) < \infty$ implies
$\mu(\{0\})=0$.} $\int_0^\infty \frac 1 s \D \mu(s)
\Le \vartheta$.
   \end{enumerate}
Moreover, if {\em (i)} holds, then the mapping
$\mm_0(\vartheta) \ni \mu \to \nu_{\mu} \in
\mm_{-1}(\vartheta)$ defined by
   \begin{align}   \label{nu}
\nu_{\mu}(\sigma) = \int_\sigma \frac 1 s \D \mu(s) +
\Big(\vartheta - \int_0^\infty \frac 1 s \D
\mu(s)\Big) \delta_0(\sigma), \quad \sigma \in
\borel{\rbb_+},
   \end{align}
is a bijection with the inverse $ \mm_{-1}(\vartheta)
\ni \nu \to \mu_{\nu} \in\mm_0(\vartheta)$ given by
   \begin{align} \label{mu}
\mu_{\nu} ( \sigma) = \int_\sigma s \D \nu (s), \quad
\sigma \in \borel{\rbb_+},
   \end{align}
where $\mm_0(\vartheta)$ stands for the set of all
representing measures $\mu$ of $\{t_n\}_{n=0}^\infty$
such that $\int_0^\infty \frac 1 s \D \mu(s) \Le
\vartheta$, and $\mm_{-1}(\vartheta)$ for the set of
all representing measures $\nu$ of
$\{t_{n-1}\}_{n=0}^\infty$. In particular,
$\nu_{\mu}(\{0\})=0$ if and only if $\int_0^\infty
\frac 1 s \D \mu(s)=\vartheta$.

If {\em (i)} holds and $\{t_n\}_{n=0}^\infty$ is
determinate, then $\{t_{n-1}\}_{n=0}^\infty$ is
determinate, the unique representing measure $\mu$ of
$\{t_n\}_{n=0}^\infty$ satisfies the inequality
$\int_0^\infty \frac 1 s \D \mu(s) \Le \vartheta$, and
$\nu_{\mu}$ is the unique representing measure of
$\{t_{n-1}\}_{n=0}^\infty$.
   \end{lem}
   \begin{proof}
Equivalence (i)$\Leftrightarrow$(ii) follows from the
Stieltjes theorem.

(iii)$\Rightarrow$(i) Clearly, if $\mu \in
\mm_0(\vartheta)$, then $t_{n-1}= \int_0^\infty s^n \D
\nu_{\mu}(s)$ for all $n \in \zbb_+$, which means that
$\{t_{n-1}\}_{n=0}^\infty$ is a Stieltjes moment sequence
and $\nu_{\mu} \in \mm_{-1}(\vartheta)$.

(i)$\Rightarrow$(iii) Take $\nu \in \mm_{-1}(\vartheta)$.
Setting $\mu:=\mu_{\nu}$ (cf.\ \eqref{mu}), we see that
   \begin{align}  \label{tnrep}
t_n = t_{(n+1)-1} = \int_0^\infty s^n s\D \nu(s) =
\int_0^\infty s^n \D \mu(s), \quad n \in \zbb_+.
   \end{align}
It is clear that $\mu(\{0\})=0$ and thus
   \begin{align*}
\int_0^\infty \frac 1 s \D \mu(s) & =
\int_{(0,\infty)} \D \nu(s) = \nu((0,\infty))
   \\
   & = \int_{[0,\infty)} s^0 \D \nu(s) - \nu(\{0\}) =
\vartheta - \nu(\{0\}),
   \end{align*}
which implies that $\int_0^\infty \frac 1 s \D \mu(s)
\Le \vartheta$. This, combined with \eqref{tnrep},
shows that $\mu \in \mm_0(\vartheta)$. Since
$\nu(\rbb_+)=\vartheta$, we deduce from \eqref{nu} and
the definition of $\mu$ that
   \allowdisplaybreaks
   \begin{align*}
\nu_{\mu}(\sigma) &= \int_{\sigma\setminus \{0\}}
\frac 1 s \D \mu(s) + \Big(\vartheta-\int_0^\infty
\frac 1 s \D \mu(s)\Big) \delta_0(\sigma \cap \{0\})
   \\
&= \nu(\sigma\setminus \{0\}) + \Big(\vartheta-
\nu((0, \infty))\Big) \delta_0(\sigma \cap \{0\})
   \\
&= \nu(\sigma\setminus \{0\}) + \nu(\{0\})
\delta_0(\sigma \cap \{0\}) = \nu(\sigma), \quad
\sigma \in \borel{\rbb_+},
   \end{align*}
which yields $\nu_{\mu} = \nu$.

We have proved that, under the assumption (i), the
mapping $\mm_0(\vartheta) \ni \mu \to \nu_{\mu} \in
\mm_{-1}(\vartheta)$ is well-defined and surjective.
Its injectivity follows from the equality
   \begin{align*}
\mu(\sigma) = \mu(\sigma \setminus \{0\}) =
\int_{\sigma \setminus \{0\}} s \D \nu_{\mu}(s), \quad
\sigma \in \borel{\rbb_+}, \, \mu \in \mm_0(\vartheta).
   \end{align*}
This yields the determinacy part of the conclusion.
   \end{proof}
   \begin{rem}
Let us discuss some consequences of Lemma
\ref{bext}. Suppose that $\{t_n\}_{n=0}^\infty$
is a determinate Stieltjes moment sequence with a
representing measure $\mu$. If $\int_0^\infty
\frac 1 s \D \mu(s) = \infty$ (e.g., when
$\mu(\{0\}) > 0$), then the sequence
$\{\vartheta, t_0, t_1, \ldots\}$ is never a
Stieltjes moment sequence. In turn, if
$\int_0^\infty \frac 1 s \D \mu(s) < \infty$,
then the sequence $\{\vartheta, t_0, t_1,
\ldots\}$ is a determinate Stieltjes moment
sequence if $\vartheta \Ge \int_0^\infty \frac 1
s \D \mu(s)$, and it is not a Stieltjes moment
sequence if $\vartheta < \int_0^\infty \frac 1 s
\D \mu(s)$.
   \end{rem}
   \begin{rem}
Under the assumptions of Lemma \ref{bext}, if
$\{t_{n-1}\}_{n=0}^\infty$ is a Stieltjes moment
sequence and $t_0 > 0$, then $t_n > 0$ for all $n
\in \zbb_+$ and
   \begin{align*}
\sup_{n \in \zbb_+}\frac{t_n^2}{t_{2n+1}} \Le
\int_0^\infty \frac{1}{s} \D \mu(s) \Le \vartheta,
\quad \mu \in \mm_0(\vartheta).
   \end{align*}
Indeed, since $t_0 > 0$ and $\mu(\{0\})=0$, we
verify that $t_n > 0$ for all $n \in \zbb_+$. By
the Cauchy-Schwarz inequality, we have
   \begin{align*}
t_n^2 = \Big(\int_{(0,\infty)}
s^{-\nicefrac12}s^{n+\nicefrac12} \D \mu(s)\Big)^2 \Le
\int_0^\infty \frac{1}{s} \D\mu(s) \int_0^\infty
s^{2n+1} \D\mu(s), \quad n \in \zbb_+.
   \end{align*}
Note that if $\{t_n\}_{n=0}^\infty$ is
indeterminate, then there is a smallest
$\vartheta$ for which the sequence
$\{t_{n-1}\}_{n=0}^\infty$ is a Stieltjes moment
sequence (see \cite{j-j-s4} for more details).
   \end{rem}
   \section{A General Setting for Subnormality}
   \subsection{\label{cfsub}Criteria for subnormality}
The only known general characterization of
subnormality of unbounded Hilbert space operators is
due to Bishop and Foia\c{s} (cf.\ \cite{bis,foi}; see
also \cite{FHSz} for a new approach via sesquilinear
selection of elementary spectral measures). Since this
characterization refers to semispectral measures (or
elementary spectral measures), it seems to be useless
in the context of weighted shifts on directed trees.
The other known criteria for subnormality require the
operator in question to have an invariant domain (with
the exception of \cite{sz4}). Since a closed subnormal
operator with an invariant domain is automatically
bounded (see \cite[Lemma 2.2(ii)]{las}, see also
\cite{ota,oka}) and a weighted shift operator $\slam$
on a directed tree is always closed (cf.\ Proposition
\ref{bas}\,(i)), we have to find a smaller subspace of
$\dz{\slam}$ which is an invariant core of $\slam$.
This will enable us to apply the aforesaid criteria
for subnormality of operators with invariant domains
in the context of weighted shift operators on directed
trees (see Section \ref{sf-a}).

We begin by recalling a characterization of
subnormality invented in \cite{c-s-sz}.
   \begin{thm}[\mbox{\cite[Theorem 21]{c-s-sz}}]
   \label{chsub} Let $S$ be a densely defined operator
in a complex Hilbert space $\hh$ such that $S(\dz S)
\subset \dz S$. Then the following conditions are
equivalent{\em :}
   \begin{enumerate}
   \item[(i)] $S$ is subnormal,
   \item[(ii)] for every finite system $\{a_{p,q}^{i,j}\}_{p,q = 0,
\ldots, n}^{i,j=1, \ldots, m} \subset \cbb$, if
   \begin{align} \label{1}
\sum_{i,j=1}^m \sum_{p,q=0}^n a_{p,q}^{i,j} \lambda^p
\bar \lambda^q z_i \bar z_j \Ge 0, \quad \lambda, z_1,
\ldots, z_m \in \cbb,
   \end{align}
then
   \begin{align*}
\sum_{i,j=1}^m \sum_{p,q=0}^n a_{p,q}^{i,j} \is{S^p
f_i} {S^q f_j} \Ge 0, \quad f_1, \ldots, f_m \in \dz
S.
   \end{align*}
   \end{enumerate}
   \end{thm}
Using the above characterization, we show that
some weak-type limit procedure preserves
subnormality (this can also be done with the help
of either \cite[Theorem 3]{StSz1} or
\cite[Theorem 37]{StSz2}; however these two
characterizations take more complicated forms).
This is a key tool for proving Theorem
\ref{main}.
   \begin{thm} \label{tw1}
Let $\{S_{\omega}\}_{\omega \in \varOmega}$ be a net
of subnormal operators in a complex Hilbert space
$\hh$ and let $S$ be a densely defined operator in
$\hh$. Suppose that there is a subset $\xx$ of $\hh$
such that
   \begin{enumerate}
   \item[(i)] $\xx \subseteq \dzn{S} \cap
\bigcap_{\omega \in \varOmega}\dzn{S_{\omega}}$,
   \item[(ii)] $\ff := \lin \bigcup_{n=0}^\infty
S^n(\xx)$ is a core of $S$,
   \item[(iii)]  $\is{S^m x}{S^n y} =
\lim_{\omega \in \varOmega} \is{S_{\omega}^m
x}{S_{\omega}^n y}$ for all $x,y \in \xx$ and $m,n \in
\zbb_+$.
   \end{enumerate}
Then $S$ is subnormal.
   \end{thm}
   \begin{proof}
Set $\ff_{\omega}=\lin \bigcup_{n=0}^\infty
S_{\omega}^n(\xx)$ for $\omega \in \varOmega$. It
is clear that $S_{\omega}|_{\ff_{\omega}}$ is a
subnormal operator in $\overline{\ff_{\omega}}$
with an invariant domain.

Take a finite system $\{a_{p,q}^{i,j}\}_{p,q = 0,
\ldots, n}^{i,j=1, \ldots, m}$ of complex numbers
satisfying \eqref{1}. Let $f_1, \ldots, f_m$ be
arbitrary vectors in $\ff$. Then for every $i \in
\{1, \ldots, m\}$, there exists a positive
integer $r$ and a system
$\{\zeta_{x,k}^{(i)}\colon x \in \xx, k= 1,
\ldots, r\}$ of complex numbers such that the set
$\{x \in \xx\colon \zeta_{x,k}^{(i)} \neq 0\}$ is
finite for every $k\in \{1, \ldots, r\}$, and
$f_i = \sum_{x \in \xx}\sum_{k=1}^r
\zeta_{x,k}^{(i)} S^k x$. Set $f_{i,\omega} =
\sum_{x \in \xx}\sum_{k=1}^r \zeta_{x,k}^{(i)}
S_{\omega}^k x$ for $i \in \{1, \ldots, m\}$ and
$\omega \in \varOmega$. Then $f_{i,\omega} \in
\ff_{\omega}$ for all $i \in \{1, \ldots, m\}$
and $\omega \in \varOmega$. Applying Theorem
\ref{chsub} to the subnormal operators
$S_{\omega}|_{\ff_{\omega}}$, we get
\allowdisplaybreaks
   \begin{multline*}
\sum_{i,j=1}^m \sum_{p,q=0}^n a_{p,q}^{i,j} \is{S^p
f_i} {S^q f_j} = \sum_{i,j=1}^m \sum_{p,q=0}^n
\sum_{x,y \in \xx} \sum_{k,l=1}^r a_{p,q}^{i,j}
\zeta_{x,k}^{(i)} \overline{\zeta_{y,l}^{(j)}}
\is{S^{p+k} x} {S^{q+l} y}
   \\
\overset{{\rm (iii)}}= \lim_{\omega \in \varOmega}
\sum_{i,j=1}^m \sum_{p,q=0}^n \sum_{x,y \in \xx}
\sum_{k,l=1}^r a_{p,q}^{i,j} \zeta_{x,k}^{(i)}
\overline{\zeta_{y,l}^{(j)}} \is{S_{\omega}^{p+k}
x} {S_{\omega}^{q+l} y}
   \\
= \lim_{\omega \in \varOmega} \sum_{i,j=1}^m
\sum_{p,q=0}^n a_{p,q}^{i,j} \is{S_{\omega}^p
f_{i,\omega}} {S_{\omega}^q f_{j,\omega}} \Ge 0.
   \end{multline*}
This means that the operator $S|_{\ff}$ satisfies
condition (ii) of Theorem \ref{chsub}. Since
$S|_{\ff}$ has an invariant domain, we deduce from
Theorem \ref{chsub} that $S|_{\ff}$ is subnormal.
Combining the latter with the assumption that $\ff$ is
a core of $S$, we see that $S$ itself is subnormal.
This completes the proof.
   \end{proof}
We say that a densely defined operator $S$ in a
complex Hilbert space $\hh$ is {\em cyclic} with
a {\em cyclic vector} $e \in \hh$ if $e \in
\dzn{S}$ and $\lin\{S^n e\colon n=0,1, \ldots\}$
is a core of $S$.
   \begin{cor}
Let $\{S_{\omega}\}_{\omega \in \varOmega}$ be a net
of subnormal operators in a complex Hilbert space
$\hh$ and let $S$ be a cyclic operator in $\hh$ with a
cyclic vector $e$ such that
   \begin{enumerate}
   \item[(i)] $e \in
\bigcap_{\omega \in \varOmega}\dzn{S_{\omega}}$,
   \item[(ii)]  $\is{S^m e}{S^n e} =
\lim_{\omega \in \varOmega} \is{S_{\omega}^m
e}{S_{\omega}^n e}$ for all $m,n \in \zbb_+$.
   \end{enumerate}
Then $S$ is subnormal.
   \end{cor}
The following fact can be proved in much the same way
as Theorem \ref{tw1}.
   \begin{pro} \label{tw1+1}
Let $S$ be a densely defined operator in a complex
Hilbert space $\hh$. Suppose that there are a family
$\{\hh_\omega\}_{\omega \in \varOmega}$ of closed
linear subspaces of $\hh$ and an upward directed
family $\{\xx_\omega\}_{\omega \in \varOmega}$ of
subsets of $\hh$ such that
   \begin{enumerate}
   \item[(i)] $\xx_\omega \subseteq
\dzn{S}$ and $S^n(\xx_\omega) \subseteq \hh_\omega$
for all $n\in \zbb_+$ and $\omega \in \varOmega$,
   \item[(ii)] $\ff_\omega:=\lin \bigcup_{n=0}^\infty
S^n(\xx_\omega)$ is dense in $\hh_\omega$ for every
$\omega \in \varOmega$,
   \item[(iii)]  $S|_{\ff_\omega}$ is  a
subnormal operator in $\hh_\omega$ for every $\omega
\in \varOmega$,
   \item[(iv)]  $\ff:=\lin \bigcup_{n=0}^\infty
S^n\big(\bigcup_{\omega \in \varOmega}
\xx_\omega\big)$ is a core of $S$.
   \end{enumerate}
Then $S$ is subnormal.
   \end{pro}
   \begin{proof}
Clearly, the families $\{\ff_\omega\}_{\omega \in
\varOmega}$ and $\{\hh_\omega\}_{\omega \in
\varOmega}$ are upward directed, $S(\ff_\omega)
\subseteq \ff_\omega$ for all $\omega \in
\varOmega$, $\ff = \bigcup_{\omega \in \varOmega}
\ff_\omega$ and $S(\ff) \subseteq \ff$. Hence, we
can argue as in the proof of Theorem~ \ref{tw1}.
   \end{proof}
   \subsection{\label{subs1}Necessity}
We begin by recalling a well-known fact that
$C^\infty$-vectors of a subnormal operator always
generate Stieltjes moment sequences.
   \begin{pro}\label{necess-gen}
If $S$ is a subnormal operator in a complex
Hilbert space $\hh$, then $\dzn{S} = \sti{S}$,
where $\sti{S}$ stands for the set of all vectors
$f \in \dzn{S}$ such that the sequence $\{\|S^n
f\|^2\}_{n=0}^\infty$ is a Stieltjes moment
sequence.
   \end{pro}
   \begin{proof}
Let $N$ be a normal extension of $S$ acting in a
complex Hilbert space $\kk \supseteq \hh$ and let $E$
be the spectral measure of $N$. Define the mapping
$\phi \colon \cbb \to \rbb_+$ by $\phi(z)=|z|^2$, $z
\in \cbb$. Since evidently $\dzn{S} \subseteq
\dzn{N}$, we deduce from the measure transport theorem
(cf.\ \cite[Theorem 5.4.10]{b-s}) that for every $f
\in \dzn{S}$,
   \begin{align*}
\|S^n f\|^2 = \|N^n f\|^2 &= \Big\|\int_{\cbb} z^n
E(\D z)f\Big\|^2
   \\
   &= \int_{\cbb} \phi(z)^n \is{E(\D z)f}{f} =
\int_0^\infty t^n \is{F(\D t)f}{f}, \quad n \in
\zbb_+,
   \end{align*}
where $F$ is the spectral measure on $\rbb_+$ given by
$F(\sigma) = E(\phi^{-1}(\sigma))$ for $\sigma \in
\borel{\rbb_+}$. This implies that $\dzn{S} \subseteq
\sti{S}$.
   \end{proof}
Note that there are closed symmetric operators
(that are always subnormal due to \cite[Theorem 1
in Appendix I.2]{a-g}) whose squares have trivial
domain (cf.\ \cite{nai,cher}).

It follows from Proposition \ref{necess-gen} that
if $S$ is a subnormal operator in a complex
Hilbert space $\hh$ with an invariant domain,
then $S$ is densely defined and $\dz{S}=\sti{S}$.
One might expect that the reverse implication
would hold as well. This is really the case for
bounded operators (cf.\ \cite{Lam}) and for some
unbounded operators that have sufficiently many
analytic vectors (cf.\ \cite[Theorem 7]{StSz1}).
In Section \ref{cfs} we show that this is also
the case for weighted shifts on directed trees
that have sufficiently many quasi-analytic
vectors (see Theorem \ref{main-0}). However, in
general, this is not the case. Indeed, one can
construct a densely defined operator $N$ in a
complex Hilbert space $\hh$ which is not
subnormal and which has the following properties
(see \cite{Cod,Sch,sto-ark}):
   \begin{gather}  \label{fn1}
N(\dz{N}) \subseteq \dz{N}, \, \dz{N} \subseteq
\dz{N^*}, \, N^*(\dz{N}) \subseteq \dz{N}
   \\
\text{ and } N^*Nf = NN^*f \text{ for all } f\in
\dz{N}. \label{fn2}
   \end{gather}
We show that for such $N$, $\dz{N}=\sti{N}$.
Indeed, by \eqref{fn1} and \eqref{fn2}, we have
   \begin{align*}
\sum_{k,l=0}^n \|N^{k+l}f\|^2 \alpha_k
\overline{\alpha_l} = \sum_{k,l=0}^n
\is{(N^*N)^{k+l}f}{f} \alpha_k \overline{\alpha_l} =
\Big\|\sum_{k=0}^n \alpha_k (N^*N)^k f\Big\|^2 \Ge 0,
   \end{align*}
for all $f \in \dz{N}$, $n \in \zbb_+$ and
$\alpha_0,\ldots, \alpha_n \in \cbb$, which means that the
sequence $\{\|N^{n}f\|^2\}_{n=0}^\infty$ is positive
definite for every $f \in \dz{N}$. Replacing $f$ by $Nf$,
we see that the sequence $\{\|N^{n+1}f\|^2\}_{n=0}^\infty$
is positive definite for every $f \in \dz{N}$. Applying the
Stieltjes theorem, we conclude that $\dz{N}=\sti{N}$.
   \section{Towards Subnormality of Weighted Shifts}
   \subsection{Powers of weighted shifts}
Let $\tcal=(V,E)$ be a directed tree. Given a family
$\{\lambda_v\}_{v \in V^\circ}$ of complex numbers, we
define the family $\{\lambda_{u\mid v}\}_{u \in V, v
\in \des{u}}$ by
   \begin{align}    \label{luv}
\lambda_{u\mid v} =
   \begin{cases}
1 & \text{ if } v=u,
   \\
\prod_{j=0}^{n-1} \lambda_{\paa^{j}(v)} & \text{ if }
v \in \dzin{n}{u}, \, n \Ge 1.
   \end{cases}
   \end{align}
Note that due to \eqref{num3} the above definition is
correct and
   \allowdisplaybreaks
   \begin{align} \label{num2}
\lambda_{u\mid w} & = \lambda_{u\mid v}\lambda_w,
\quad w \in \dzi{v}, \, v\in \des{u}, \, u \in V,
   \\
\lambda_{\pa v\mid w} & = \lambda_v \lambda_{v\mid w},
\quad v \in V^\circ, \, w\in \des v.
   \label{recfor2}
   \end{align}
The following lemma is a generalization of
\cite[Lemma 6.1.1]{j-j-s} to the case of
unbounded operators. Below, we maintain our
general convention that $\sum_{v\in\varnothing}
x_v=~0$.
   \begin{lem} \label{lem4}
Let $\slam $ be a weighted shift on a directed
tree $\tcal$ with weights $\lambdab =
\{\lambda_v\}_{v \in V^\circ}$. Fix $u \in V$ and
$n \in \zbb_+$. Then the following assertions
hold\/{\em :}
   \begin{enumerate}
   \item[(i)] $e_u \in \dz{\slam^n}$ if and only if
$\sum_{v \in \dzin{m}{u}} |\lambda_{u\mid v}|^2 <
\infty$ for all integers $m$ such that $1 \Le m
\Le n$,
   \item[(ii)] if $e_u \in \dz{\slam^n}$, then
$\slam^n e_u = \sum_{v \in \dzin{n}{u}} \lambda_{u\mid
v} \, e_v$,
   \item[(iii)] if $e_u \in \dz{\slam^n}$, then
$\|\slam^n e_u\|^2 = \sum_{v \in \dzin{n}{u}}
|\lambda_{u\mid v}|^2$.
   \end{enumerate}
   \end{lem}
   \begin{proof} For $k \in \zbb_+$,
we define the complex function $\lambdabxx{k}{\cdot}$
on $V$ by
   \begin{align}  \label{deflam}
\lambdabxx{k}{v} =
   \begin{cases}
\lambda_{u|v} & \text{ if } v \in \dzin{k}{u},
   \\
0 & \text{ if } v \in V \setminus \dzin{k}{u}.
   \end{cases}
   \end{align}
We shall prove that for every $k\in \zbb_+$,
   \begin{gather}  \label{small}
e_u \in \dz{\slam^k} \text{ if and only if } \sum_{v
\in \dzin{m}{u}} |\lambda_{u\mid v}|^2 < \infty \text{
for } m = 0,1, \ldots, k,
   \end{gather}
   and
   \begin{gather}
\label{small2} \text{if } e_u \in \dz{\slam^k}, \text{
then } \slam^k e_u = \lambdabxx{k}{\cdot}.
   \end{gather}
We use an induction on $k$. The case of $k=0$ is
obvious. Suppose that \eqref{small} and \eqref{small2}
hold for all nonnegative integers less than or equal
to $k$. Assume that $e_u \in \dz{\slam^k}$. Now we
compute $\varLambda_\tcal (\slam^k e_u)$. It follows
from the induction hypothesis and \eqref{deflam} that
   \allowdisplaybreaks
   \begin{align*}
(\varLambda_\tcal (\slam^k e_u))(v)
&\overset{\eqref{lamtauf}}=
   \begin{cases}
\lambda_v (\slam^k e_u)(\paa(v)) & \text{ if } v \in
V^\circ,
   \\
0 & \text{ if } v=\koo,
   \end{cases}
   \\
&\overset{\eqref{small2}}=
   \begin{cases}
\lambda_v \lambdabxx{k}{\paa(v)} & \text{ if } \paa(v)
\in \dzin{k}{u},
   \\
0 & \text{ otherwise,}
   \end{cases}
   \\
& \overset{\eqref{num4}}= \begin{cases} \lambda_v
\lambda_{u|\paa(v)} & \text{ if } v \in \dzin{k+1}{u},
   \\
0 & \text{ otherwise,}
   \end{cases}
   \\
& \overset{\eqref{num2}}= \begin{cases} \lambda_{u|v}
& \text{ if } v \in \dzin{k+1}{u},
   \\
0 & \text{ otherwise,}
   \end{cases}
   \\
&\hspace{1.7ex} = \lambdabxx{k+1}{v}, \quad v \in V,
  \end{align*}
which shows that $\varLambda_\tcal (\slam^k e_u) =
\lambdabxx{k+1}{\cdot}$. This in turn implies that
\eqref{small} and \eqref{small2} hold for $k+1$ in
place of $k$. This proves (i) and (ii). Assertion
(iii) is a direct consequence of (ii).
   \end{proof}
In the context of weighted shifts on directed
trees, the key assumption (iii) of Theorem
\ref{tw1} can be verified by using the following
relatively simple criterion that may be of
independent interest.
   \begin{pro}\label{potegi}
If $\lambdabi =\big\{\lambdai{v}\big\}_{v \in
V^\circ}$, $i=1,2,3, \ldots$, and
$\lambdab=\{\lambda_v\}_{v\in V^\circ}$ are families
of complex numbers such that
   \begin{enumerate}
   \item[(i)] $\escr \subseteq \dzn{\slam}
\cap \bigcap_{i=1}^\infty \dzn{S_{\lambdabi}}$,
   \item[(ii)] $\lim_{i \to \infty} \lambdai{v} =
\lambda_v$ for all $v \in V^\circ$,
   \item[(iii)] $\lim_{i \to \infty}
\|S_{\lambdabi}^n e_u\| = \|\slam^n e_u\|$ for all $n
\in \zbb_+$ and $u \in V$,
   \end{enumerate}
then
   \begin{align}  \label{slim+}
\is{\slam^m e_u}{\slam^n e_v} = \lim_{i \to \infty}
\is{S_{\lambdabi}^m e_u}{S_{\lambdabi}^n e_v}, \quad
u,v \in V, \, m,n \in \zbb_+.
   \end{align}
   \end{pro}
   \begin{proof} We split the proof into two steps.

{\sc Step 1.} If $\lambdab=\{\lambda_v\}_{v\in
V^\circ}$ is a family of complex numbers such that
$\escr \subseteq \dzn{\slam}$, then for all $m,n \in
\zbb_+$ and $u,v \in V$,
   \begin{align} \label{smsn}
\is{\slam^m e_u}{\slam^n e_v} =
   \begin{cases}
0 & \text{ if } \cc^{m,n}(u,v) = \varnothing,
   \\
\overline{\lambda_{v|u}} \, \|\slam^m e_u\|^2 &
\text{ if } \cc^{m,n}(u,v) \neq \varnothing \text{
and } m\Le n,
   \\
\lambda_{u|v} \, \|\slam^n e_v\|^2 & \text{ if }
\cc^{m,n}(u,v) \neq \varnothing \text{ and } m >n,
   \end{cases}
   \end{align}
where $\cc^{m,n}(u,v) := \dzin{m}{u} \cap
\dzin{n}{v}$.

Indeed, it follows from Lemma \ref{lem4} that
   \begin{align}
   \begin{aligned}  \label{slmsln}
\is{\slam^m e_u}{\slam^n e_v} & =
\Big\langle\sum_{u^\prime \in \dzin{m}{u}}
\lambda_{u\mid u^\prime} \, e_{u^\prime},
\sum_{v^\prime \in \dzin{n}{v}} \lambda_{v\mid
v^\prime} \, e_{v^\prime}\Big\rangle
   \\
& = \sum_{u^\prime \in \cc^{m,n}(u,v)} \lambda_{u\mid
u^\prime} \overline{\lambda_{v\mid u^\prime}}.
   \end{aligned}
   \end{align}
Hence, if $\cc^{m,n}(u,v) = \varnothing$, then
the left-hand side of \eqref{smsn} is equal to
$0$ as required. Suppose now that $\cc^{m,n}(u,v)
\neq \varnothing$ and $m \Le n$. Then
   \begin{align} \label{num8}
\cc^{m,n}(u,v)=\dzin{m}{u}.
   \end{align}
To show this, take $w \in \cc^{m,n}(u,v)$. Then, by
\eqref{num4}, $u=\paa^m(w)$ and
   \begin{align*}
v = \paa^{n}(w) = \paa^{n-m}(\paa^m(w)) =
\paa^{n-m}(u),
   \end{align*}
which, by \eqref{num4} again, is equivalent to
   \begin{align} \label{num7}
u \in \dzin{n-m}{v}.
   \end{align}
This implies that
   \begin{align} \label{num6}
\dzin{m}{u} \subseteq \dzin{m}{\dzin{n-m}{v}}
\overset{\eqref{chmn}}{=} \dzin{n}{v}.
   \end{align}
Thus \eqref{num8} holds. Next, we show that
   \begin{align} \label{num5}
\lambda_{v\mid u^\prime} = \lambda_{u\mid u^\prime}
\lambda_{v|u}, \quad u^\prime \in \dzin{m}{u}.
   \end{align}
It is enough to consider the case where $m\Ge 1$
and $n > m$. Since $u^\prime \in \dzin{m}{u}$, we
infer from \eqref{num6} that $u^\prime \in
\dzin{n}{v}$. Moreover, by \eqref{num7}, $u \in
\dzin{n-m}{v}$. All these facts together with
\eqref{luv} imply that
   \begin{multline*}
\lambda_{v|u^\prime} = \prod_{j=0}^{n-1}
\lambda_{\paa^{j}(u^\prime)} = \prod_{j=0}^{m-1}
\lambda_{\paa^{j}(u^\prime)} \prod_{j=m}^{n-1}
\lambda_{\paa^{j}(u^\prime)}
   \\
\overset{\eqref{luv}}= \lambda_{u\mid u^\prime}
\prod_{j=0}^{n-m-1}
\lambda_{\paa^{j}(\paa^{m}(u^\prime))}
\overset{\eqref{num4}}= \lambda_{u\mid u^\prime}
\prod_{j=0}^{n-m-1} \lambda_{\paa^{j}(u)}
\overset{\eqref{luv}}= \lambda_{u\mid u^\prime}
\lambda_{v|u},
   \end{multline*}
which completes the proof of \eqref{num5}. Now
applying \eqref{slmsln}, \eqref{num8}, \eqref{num5}
and Lemma \ref{lem4}\,(iii), we obtain
   \allowdisplaybreaks
   \begin{align*}
\is{\slam^m e_u}{\slam^n e_v} & = \sum_{u^\prime \in
\dzin{m}{u}} \lambda_{u\mid u^\prime}
\overline{\lambda_{v\mid u^\prime}}
   \\
& \hspace{-2.2ex} \overset{\eqref{num5}}=
\overline{\lambda_{v|u}} \sum_{u^\prime \in
\dzin{m}{u}} |\lambda_{u\mid u^\prime}|^2 =
\overline{\lambda_{v|u}} \, \|\slam^m e_u\|^2.
   \end{align*}
Taking the complex conjugate and making appropriate
substitutions, we infer from the above that
$\is{\slam^m e_u}{\slam^n e_v} = \lambda_{u|v} \,
\|\slam^n e_v\|^2$ if $\cc^{m,n}(u,v) \neq
\varnothing$ and $m
>n$, which completes the proof of Step 1.

{\sc Step 2.} Under the assumptions of Proposition
\ref{potegi}, equality \eqref{slim+} holds.

Indeed, it follows from (ii) that
   \begin{align} \label{wzj+}
\lim_{i \to \infty} \lambdai{u\mid v} =
\lambda_{u\mid v}, \quad u \in V, v \in \des{u},
   \end{align}
where $\{\lambdai{u\mid v}\}_{u \in V, v \in
\des{u}}$ is the family related to
$\big\{\lambdai{v}\big\}_{v \in V^\circ}$ via
\eqref{luv}. Now, applying Step 1 to the
operators $S_{\lambdabi}$ and $\slam$ (which is
possible due to (i)) and using \eqref{wzj+} and
(iii), we obtain \eqref{slim+}.
   \end{proof}
   \subsection{A consistency condition}
The following is an immediate consequence of
Proposition \ref{necess-gen}.
   \begin{pro}\label{necess}
Let $\slam$ be a weighted shift on a directed tree
$\tcal$ with weights $\lambdab=\{\lambda_v\}_{v \in
V^\circ}$ such that $\escr \subseteq \dzn{\slam}$. If
$\slam$ is subnormal, then for every $u \in V$ the
sequence $\{\|\slam^n e_u\|^2\}_{n=0}^\infty$ is a
Stieltjes moment sequence.
   \end{pro}
The converse of the implication in Proposition
\ref{necess} is valid for bounded weight\-ed shifts on
directed trees.
   \begin{thm}[\mbox{\cite[Theorem 6.1.3]{j-j-s}}]
\label{charsub}
   Let $\slam \in \ogr{\ell^2(V)}$ be a weighted shift
on a directed tree $\tcal$ with weights $\lambdab =
\{\lambda_v\}_{v \in V^\circ}$. Then $\slam$ is
subnormal if and only if $\{\|\slam^n
e_u\|^2\}_{n=0}^\infty$ is a Stieltjes moment sequence
for every $u \in V$.
   \end{thm}
   The case of unbounded weighted shifts is discussed
in Theorem \ref{main-0}.

If $\slam$ is a subnormal weighted shift on a directed
tree $\tcal$, then in view of Proposition \ref{necess}
we can attach to each vertex $u \in V$ a representing
measure $\mu_u$ of the Stieltjes moment sequence
$\{\|\slam^n e_u\|^2\}_{n=0}^\infty$ (of course, since
the sequence $\{\|\slam^n e_u\|^2\}_{n=0}^\infty$ is
not determinate in general, we have to choose one of
them); note that any such $\mu_u$ is a probability
measure. Hence, it is tempting to find relationships
between these representing measures. This has been
done in the case of bounded weighted shifts in
\cite[Lemma 6.1.10]{j-j-s}. What is stated below is an
adaptation of this lemma (and its proof) to the
unbounded case. As opposed to the bounded case,
implication $1^\circ \Rightarrow 2^\circ$ of Lemma
\ref{charsub2} below is not true in general (cf.\
\cite{j-j-s4}).
   \begin{lem} \label{charsub2}
Let $\slam$ be a weighted shift on a directed tree
$\tcal$ with weights $\lambdab=\{\lambda_v\}_{v \in
V^\circ}$ such that $\escr \subseteq \dzn{\slam}$. Let
$u \in V^\prime$. Suppose that for every $v \in \dzi u$
the sequence $\{\|\slam^n e_v\|^2\}_{n=0}^\infty$ is a
Stieltjes moment sequence with a representing measure
$\mu_v$. Consider the following two
conditions\,\footnote{\;\;We adhere to the standard
convention that $0 \cdot \infty = 0$; see also footnote
\ref{foot}.}{\em :}
   \begin{enumerate}
   \item[$1^\circ$] $\{\|\slam^n e_u\|^2\}_{n=0}^\infty$
is a Stieltjes moment sequence,
   \item[$2^\circ$]   $\slam$ satisfies
the consistency condition at the vertex $u$, i.e.,
   \begin{align} \label{alanconsi}
\sum_{v \in \dzi{u}} |\lambda_v|^2 \int_0^\infty \frac
1 s\, \D \mu_v(s) \Le 1.
   \end{align}
   \end{enumerate}
   Then the following assertions are valid{\em :}
   \begin{enumerate}
   \item[(i)] if $2^\circ$ holds, then so does
$1^\circ$ and the positive Borel measure $\mu_u$ on
$\rbb_+$ defined by
   \begin{align} \label{muu+}
\mu_u(\sigma) = \sum_{v \in \dzi u} |\lambda_v|^2
\int_\sigma \frac 1 s \D \mu_v(s) + \varepsilon_u
\delta_0(\sigma), \quad \sigma \in \borel{\rbb_+},
      \end{align}
with
   \begin{align} \label{muu++}
\varepsilon_u=1 - \sum_{v \in \dzi u} |\lambda_v|^2
\int_0^\infty \frac 1 s \D \mu_v(s)
   \end{align}
is a representing measure of $\{\|\slam^n
e_u\|^2\}_{n=0}^\infty$,
   \item[(ii)] if $1^\circ$ holds and $\{\|\slam^{n+1}
e_u\|^2\}_{n=0}^\infty$ is determinate, then $2^\circ$
holds, the Stieltjes moment sequence $\{\|\slam^n
e_u\|^2\}_{n=0}^\infty$ is determinate and its unique
representing measure $\mu_u$ is given by \eqref{muu+}
and \eqref{muu++}.
   \end{enumerate}
   \end{lem}
   \begin{proof}
Define the positive Borel measure $\mu$ on $\rbb_+$ by
   \begin{align*}
\mu(\sigma) = \sum_{v \in \dzi u} |\lambda_v|^2
\mu_v(\sigma), \quad \sigma \in \borel{\rbb_+}.
   \end{align*}
It is a matter of routine to show that
   \begin{align}  \label{leb2}
\int_0^\infty f \D \mu = \sum_{v \in \dzi u}
|\lambda_v|^2 \int_0^\infty f \D \mu_v
   \end{align}
for every Borel function $f\colon {[0,\infty)}
\to [0,\infty]$. Using the inclusion $\escr
\subseteq \dzn{\slam}$ and applying Lemma
\ref{lem4}\,(iii) twice, we obtain
   \allowdisplaybreaks
   \begin{align*}
\|\slam^{n+1} e_u\|^2 & \hspace{2.2ex}= \sum_{w \in
\dzin{n+1}{u}} |\lambda_{u\mid w}|^2
   \\
& \hspace{.4ex} \overset{\eqref{dzinn2}}= \sum_{v \in
\dzi u} \sum_{w \in \dzin{n} v} |\lambda_{u\mid w}|^2
\notag
   \\
& \hspace{.4ex} \overset{\eqref{recfor2}}= \sum_{v \in
\dzi u} |\lambda_v|^2 \sum_{w \in \dzin{n} v}
|\lambda_{v\mid w}|^2 \notag
   \\
& \hspace{2.2ex} =\sum_{v \in \dzi u} |\lambda_v|^2
\|\slam^n e_v\|^2, \quad n \in \zbb_+. \notag
   \end{align*}
This implies that
   \begin{align*}
\|\slam^{n+1} e_u\|^2 = \sum_{v \in \dzi u}
|\lambda_v|^2 \int_0^\infty s^n \, \D \mu_v(s)
\overset{\eqref{leb2}}= \int_0^\infty s^n \D \mu(s),
\quad n \in \zbb_+.
   \end{align*}
   Hence the sequence $\{\|\slam^{n+1}
e_u\|^2\}_{n=0}^\infty$ is a Stieltjes moment sequence
with a representing measure $\mu$.

Set $t_n = \|\slam^{n+1} e_u\|^2$ for $n \in \zbb_+$,
and $t_{-1}=1$. Note that
   \begin{align*}
t_{n-1}=\|\slam^{n} e_u\|^2, \quad n \in \zbb_+.
   \end{align*}

Suppose that $2^\circ$ holds. Then, by
\eqref{alanconsi} and \eqref{leb2}, we have
$\int_0^\infty \frac 1 s \D \mu(s) \Le 1$.
Applying implication (iii)$\Rightarrow$(i) of
Lemma \ref{bext}, we see that $1^\circ$ holds,
and, by \eqref{leb2}, the measure $\mu_u$ defined
by \eqref{muu+} and \eqref{muu++} is a
representing measure of the Stieltjes moment
sequence $\{\|\slam^n e_u\|^2\}_{n=0}^\infty$.

Suppose now that $1^\circ$ holds and the
Stieltjes moment sequence $\{\|\slam^{n+1}
e_u\|^2\}_{n=0}^\infty$ is determinate. It
follows from implication (i)$\Rightarrow$(iii) of
Lemma \ref{bext} that there is a representing
measure $\mu^\prime$ of $\{\|\slam^{n+1}
e_u\|^2\}_{n=0}^\infty$ such that $\int_0^\infty
\frac 1 s \D \mu^\prime(s) \Le 1$. Since
$\{\|\slam^{n+1} e_u\|^2\}_{n=0}^\infty$ is
determinate, we get $\mu^\prime=\mu$, which
implies $2^\circ$. The remaining part of
assertion (ii) follows from the last assertion of
Lemma \ref{bext}.
   \end{proof}
Now we prove that the determinacy of appropriate
Stieltjes moment sequences attached to a weighted
shift on a directed tree implies the existence of a
consistent system of measures (see also Corollary
\ref{necessdet2}). As shown in \cite{j-j-s4}, Lemma
\ref{2necess+} below is no longer true if the
assumption on determinacy is dropped (though, by Lemma
\ref{lem3}\,(iv), the converse of Lemma \ref{2necess+}
is true without assuming determinacy).
   \begin{lem} \label{2necess+}
 Let $\slam$ be a weighted shift on a directed tree
$\tcal$ with weights $\lambdab=\{\lambda_v\}_{v \in
V^\circ}$ such that $\escr \subseteq \dzn{\slam}$.
Assume that for every $u \in V^\prime$, the sequence
$\{\|\slam^{n} e_u\|^2\}_{n=0}^\infty$ is a Stieltjes
moment sequence, and that the Stieltjes moment
sequence\,\footnote{\;see \eqref{st+1}}
$\{\|\slam^{n+1} e_u\|^2\}_{n=0}^\infty$ is
determinate. Then there exist a system $\{\mu_u\}_{u
\in V}$ of Borel probability measures on $\rbb_+$ and
a system $\{\varepsilon_u\}_{u \in V}$ of nonnegative
real numbers that satisfy \eqref{muu+} for every $u
\in V$.
   \end{lem}
   \begin{proof}
By Lemma \ref{bext}, the Stieltjes moment
sequence $\{\|\slam^n e_u\|^2\}_{n=0}^\infty$ is
determinate for every $u\in V^\prime$. For $u \in
V^\prime$, we denote by $\mu_u$ the unique
representing measure of $\{\|\slam^n
e_u\|^2\}_{n=0}^\infty$. If $u \in V \setminus
V^\prime$, then we put $\mu_u=\delta_0$. Using
Lemma \ref{charsub2}\,(ii), we verify that the
system $\{\mu_u\}_{u \in V}$ satisfies
\eqref{muu+} with $\{\varepsilon_u\}_{u \in V}$
defined by \eqref{muu++}. This completes the
proof.
   \end{proof}
   \subsection{A hereditary property} Given a weighted
shift $\slam$ on $\tcal$, we say that a vertex $u
\in V$ {\em generates} a Stieltjes moment
sequence (with respect to $\slam$) if $e_u \in
\dzn{\slam}$ and the sequence $\{\|\slam^n
e_u\|^2\}_{n=0}^\infty$ is a Stieltjes moment
sequence. We have shown in Lemma \ref{charsub2}
that in many cases the parent generates a
Stieltjes moment sequence whenever its children
do so. If the parent generates a Stieltjes moment
sequence, then in general its children do not do
so (cf.\ \cite[Example 6.1.6]{j-j-s}). However,
if the parent has only one child and generates a
Stieltjes moment sequence, then its child does
so.
   \begin{lem} \label{charsub-1}
Let $\slam$ be a weighted shift on a directed tree
$\tcal$ with weights $\lambdab = \{\lambda_v\}_{v \in
V^\circ}$ and let $u_0, u_1 \in V$ be such that
$\dzi{u_0} = \{u_1\}$. Suppose that $e_{u_0} \in
\dzn{\slam}$, $\{\|\slam^n e_{u_0}\|^2\}_{n=0}^\infty$
is a Stieltjes moment sequence and $\lambda_{u_1}\neq
0$. Then $e_{u_1} \in \dzn{\slam}$ and $\{\|\slam^n
e_{u_1}\|^2\}_{n=0}^\infty$ is a Stieltjes moment
sequence. Moreover, the following assertions
hold\/{\em :}
   \begin{enumerate}
   \item[(i)] the mapping
$\mm_{u_1}^{\mathrm b}(\lambdab) \ni \mu \to
\rho_{\mu} \in \mm_{u_0}(\lambdab)$ defined by
   \begin{align*}
\rho_{\mu}(\sigma) = |\lambda_{u_1}|^2 \int_\sigma
\frac 1 s \D \mu(s) + \Big(1 - |\lambda_{u_1}|^2
\int_0^\infty \frac 1 s \D \mu(s)\Big)
\delta_0(\sigma), \quad \sigma \in \borel{\rbb_+},
   \end{align*}
is a bijection with the inverse $\mm_{u_0}(\lambdab)
\ni \rho \to \mu_{\rho} \in \mm_{u_1}^{\mathrm
b}(\lambdab)$ given by
   \begin{align*}
\mu_{\rho} ( \sigma) = \frac 1 {|\lambda_{u_1}|^2}
\int_\sigma s \D \rho (s), \quad \sigma \in
\borel{\rbb_+},
   \end{align*}
where $\mm_{u_1}^{\mathrm b}(\lambdab)$ is the set of
all representing measures $\mu$ of $\{\|\slam^n
e_{u_1}\|^2\}_{n=0}^\infty$ such that $\int_0^\infty
\frac 1 s \D \mu(s) \Le \frac{1}{|\lambda_{u_1}|^2}$,
and $\mm_{u_0}(\lambdab)$ is the set of all
representing measures $\rho$ of $\{\|\slam^n
e_{u_0}\|^2\}_{n=0}^\infty$,
   \item[(ii)] if the Stieltjes moment
sequence $\{\|\slam^n e_{u_1}\|^2\}_{n=0}^\infty$ is
determinate, then so are $\{\|\slam^{n+1}
e_{u_0}\|^2\}_{n=0}^\infty$ and $\{\|\slam^n
e_{u_0}\|^2\}_{n=0}^\infty$.
   \end{enumerate}
   \end{lem}
   \begin{proof}  Since $e_{u_0} \in
\dzn{\slam}$, $\dzi{u_0} = \{u_1\}$ and $\lambda_{u_1}
\neq 0$, we infer from \eqref{eu} that $e_{u_1} =
\frac{1}{\lambda_{u_1}} \slam e_{u_0} \in \dzn{\slam}$
and thus
  \begin{align*}
\|\slam^n e_{u_1}\|^2 = \frac 1
{|\lambda_{u_1}|^2}\|\slam^{n+1} e_{u_0}\|^2, \quad n
\in \zbb_+.
   \end{align*}
The last equality and Lemma \ref{bext} applied to
$\vartheta=1$ and $t_n=\|\slam^{n+1} e_{u_0}\|^2$ ($n
\in \zbb_+$) complete the proof.
   \end{proof}
   \section{Criteria for Subnormality of
Weighted Shifts}
   \subsection{Consistent systems of measures}
In this section we prove some important
properties of consistent systems of Borel
probability measures on $\rbb_+$ attached to a
directed tree. They will be used in the proof of
Theorem \ref{main}.
   \begin{lem} \label{lem1}
Let $\tcal$ be a directed tree. Suppose that
$\{\lambda_v\}_{v \in V^\circ}$ is a system of
complex numbers, $\{\varepsilon_v\}_{v \in V}$ is
a system of nonnegative real numbers and
$\{\mu_v\}_{v \in V}$ is a system of Borel
probability measures on $\rbb_+$ satisfying
\eqref{muu+} for every $u \in V$. Then the
following assertions hold\/{\em :}
   \begin{enumerate}
   \item[(i)] for every $u \in V$,
$\sum_{v\in \dzi{u}} |\lambda_v|^2 \int_0^\infty \frac
1 s \D \mu_v(s) \Le 1$ and
   \begin{align*}
\varepsilon_u = 1 - \sum_{v\in \dzi{u}} |\lambda_v|^2
\int_0^\infty \frac 1 s \D \mu_v(s),
   \end{align*}
   \item[(ii)] for every $u \in V$,  $\mu_u(\{0\})= 0$
if and only if $\varepsilon_u=0$,
   \item[(iii)] for every $v \in V^\circ$, if $\lambda_v \neq
0$, then $\mu_v(\{0\})=0$,
   \item[(iv)]  for every $u \in V$,
   \begin{align}  \label{wz2}
\mu_u(\sigma) = \sum_{v\in \dzin{n}{u}}
|\lambda_{u\mid v}|^2 \int_{\sigma} \frac 1 {s^n} \D
\mu_v(s) + \varepsilon_u \delta_0(\sigma), \quad
\sigma \in \borel{\rbb_+}, \, n \Ge 1.
   \end{align}
   \end{enumerate}
   \end{lem}
   \begin{proof}
   (i) Substitute $\sigma = \rbb_+$ into \eqref{muu+}
and note that $\mu_u(\rbb_+)=1$.

   (ii) \& (iii) Substitute $\sigma = \{0\}$ into
\eqref{muu+}.

   (iv) We use induction on $n$. The case of $n=1$
coincides with \eqref{muu+}. Suppose that \eqref{wz2}
is valid for a fixed integer $n\Ge 1$. Then combining
\eqref{muu+} with \eqref{wz2}, we see that
   \begin{multline} \label{wz3}
\mu_u(\sigma) = \sum_{v\in \dzin{n}{u}}
|\lambda_{u\mid v}|^2 \sum_{w \in \dzi{v}}
|\lambda_w|^2\int_{\sigma} \frac 1 {s^{n+1}} \D
\mu_w(s)
   \\
   + \sum_{v\in \dzin{n}{u}} |\lambda_{u\mid v}|^2
\int_{\sigma} \frac 1 {s^n} \D (\varepsilon_v
\delta_0)(s) + \varepsilon_u \delta_0(\sigma), \quad
\sigma \in \borel{\rbb_+}.
   \end{multline}
Since $\mu_u$ is a finite positive measure and
$n\Ge 1$, we deduce from \eqref{wz3} that
$\varepsilon_v=0$ whenever $\lambda_{u\mid v}
\neq 0$, and thus
   \begin{align}  \label{wz4}
\sum_{v\in \dzin{n}{u}} |\lambda_{u\mid v}|^2
\int_{\sigma} \frac 1 {s^n} \D (\varepsilon_v
\delta_0)(s)=0.
   \end{align}
It follows from \eqref{wz3} and \eqref{wz4} that
   \begin{align*}
\mu_u(\sigma) = \sum_{v\in \dzin{n}{u}} \sum_{w \in
\dzi{v}} |\lambda_{u\mid v}\lambda_w|^2\int_{\sigma}
\frac 1 {s^{n+1}} \D \mu_w(s) + \varepsilon_u
\delta_0(\sigma)
      \\
      \overset{\eqref{num1}\&\eqref{num2}}= \sum_{w\in
\dzin{n+1}{u}} |\lambda_{u\mid w}|^2 \int_{\sigma}
\frac 1 {s^{n+1}} \D \mu_w(s) + \varepsilon_u
\delta_0(\sigma).
   \end{align*}
This completes the proof.
   \end{proof}
   \begin{lem} \label{lem3}
Let $\tcal$ be a directed tree. Suppose that
$\lambdab=\{\lambda_v\}_{v \in V^\circ}$ is a
system of complex numbers, $\{\varepsilon_v\}_{v
\in V}$ is a system of nonnegative real numbers
and $\{\mu_v\}_{v \in V}$ is a system of Borel
probability measures on $\rbb_+$ satisfying
\eqref{muu+} for every $u \in V$. Let $\slam$ be
a weighted shift on the directed tree $\tcal$
with weights $\lambdab$. Then the following
assertions hold\/{\em :}
   \begin{enumerate}
   \item[(i)] for all $u \in V$ and $n \in \nbb$,
   \begin{align}  \label{wz5}
\int_0^\infty s^n \D \mu_u(s) = \sum_{v\in
\dzin{n}{u}} |\lambda_{u\mid v}|^2,
   \end{align}
   \item[(ii)] if $\dzin{n}{u}=\varnothing$ for some
$u\in V$ and $n\in \nbb$, then $\mu_v=\delta_0$ for
all $v \in \des{u}$,
   \item[(iii)] $\escr \subseteq \dzn{\slam}$ if and only
if $\int_0^\infty s^n \D \mu_u(s) < \infty$ for all $n
\in \zbb_+$ and $u \in V$,
   \item[(iv)] if $\escr \subseteq \dzn{\slam}$, then
for all $u \in V$ and $n \in \zbb_+$,
   \begin{align}    \label{wz6}
\|\slam^n e_u\|^2 = \int_0^\infty s^n \D \mu_u(s),
   \end{align}
   \item[(v)] $\slam \in \ogr{\ell^2(V)}$ if and
only if there exists a real number $M \Ge 0$ such that
$\supp \mu_u \subseteq [0,M]$ for every $u \in V$.
   \end{enumerate}
   \end{lem}
   \begin{proof}
   (i) Substituting $\sigma=\{0\}$ into
\eqref{wz2}, we see that for every $v\in
\dzin{n}{u}$, either $\lambda_{u\mid v} = 0$, or
$\lambda_{u\mid v} \neq 0$ and $\mu_v(\{0\})=0$.
This and \eqref{wz2} lead to \eqref{wz5}.

   (ii) It follows from \eqref{wz5} that
$\int_0^\infty s^n \D \mu_u(s) = 0$ (recall the
convention that $\sum_{v\in\varnothing} x_v=0$). This
and $n \Ge 1$ implies that $\mu_u((0,\infty))=0$.
Since $\mu_u(\rbb_+) = 1$, we deduce that
$\mu_u=\delta_0$.

If $v \in \des{u}\setminus \{u\}$, then by
\eqref{num3} there exists $k\in \nbb$ such that
$v\in \dzin{k}{u}$. Since $\dzi{\cdot}$ is a
monotonically increasing set-function, we infer
from \eqref{chmn} that $\dzin{n}{v} \subseteq
\dzin{n+k}{u}=\varnothing$. By the previous
argument applied to $v$ in place of $u$, we get
$\mu_v=\delta_0$.

   Assertions (iii) and (iv) follow from (i) and Lemma
\ref{lem4}.

   (v) To prove the ``only if'' part, note that
   \begin{align*}
\lim_{n\to\infty} \Big(\int_0^\infty s^{n} \D
\mu_u(s)\Big)^{1/n} \overset{\eqref{wz6}}=
\lim_{n\to\infty} (\|\slam^n e_u\|^{1/n})^2 \Le
\|\slam\|^2,
   \end{align*}
which implies that $\supp \mu_u \subseteq
[0,\|\slam\|^2]$ (cf.\ \cite[page 71]{Rud}). The proof
of the converse implication goes as follows. By
\eqref{wz5}, we have
   \begin{align*}
\sum_{v\in \dzi{u}} |\lambda_{v}|^2 = \int_0^\infty s
\D \mu_u(s) \Le M, \quad u \in V,
   \end{align*}
which in view of Proposition \ref{bas}\,(v) implies
that $\slam \in \ogr{\ell^2(V)}$ and $\|\slam\| \Le
\sqrt{M}$. This completes the proof.
   \end{proof}
   \subsection{\label{sf-a}Arbitrary weights}
   After all these preparations we can prove the main
criterion for subnormality of unbounded weighted
shifts on directed trees. It is written in terms of
consistent systems of measures.
   \begin{thm} \label{main}
 Let $\slam$ be a weighted shift on a directed
tree $\tcal$ with weights
$\lambdab=\{\lambda_v\}_{v \in V^\circ}$ such
that $\escr \subseteq \dzn{\slam}$. Suppose that
there exist a system $\{\mu_v\}_{v \in V}$ of
Borel probability measures on $\rbb_+$ and a
system $\{\varepsilon_v\}_{v \in V}$ of
nonnegative real numbers that satisfy
\eqref{muu+} for every $u \in V$. Then $\slam$ is
subnormal.
   \end{thm}
   \begin{proof}
   For a fixed positive integer $i$, we define the
system $\lambdabi =\big\{\lambdai{v}\big\}_{v \in
V^\circ}$ of complex numbers, the system
$\big\{\mui{v}\big\}_{v \in V}$ of Borel probability
measures on $\rbb_+$ and the system
$\big\{\varepsiloni{v}\big\}_{v \in V}$ of nonnegative
real numbers by
   \allowdisplaybreaks
   \begin{align}  \label{wzd1}
\lambdai{v} & =
   \begin{cases}
   \lambda_v
\sqrt{\cfrac{\mu_v([0,i])}{\mu_{\pa{v}}([0,i])}} &
\text{ if } \mu_{\pa{v}}([0,i]) > 0,
   \\[1.5ex]
0 & \text{ if } \mu_{\pa{v}}([0,i]) = 0,
   \end{cases}
\quad v \in V^\circ,
   \\     \label{wzd2}
\mui{v}(\sigma) & =
   \begin{cases}
\cfrac{\mu_v(\sigma \cap [0,i])}{\mu_v([0,i])} &
\text{ if } \mu_v([0,i]) > 0,
   \\[1.5ex]
\delta_0(\sigma) & \text{ if } \mu_v([0,i]) = 0,
   \end{cases}
\quad \sigma \in \borel{\rbb_+},\, v \in V,
   \\  \label{wzd3}
\varepsiloni{v} & =
   \begin{cases}
\cfrac{\varepsilon_v}{\mu_v([0,i])} & \text{ if }
\mu_v([0,i]) > 0,
   \\[1.5ex]
1 & \text{ if } \mu_v([0,i])=0,
   \end{cases}
\quad v \in V.
   \end{align}
Our first goal is to show that the following equality
holds for all $u \in V$ and $i\in \nbb$,
   \begin{align} \label{wz1J}
\mui{u}(\sigma) = \sum_{v\in \dzi{u}} |\lambdai{v}|^2
\int_{\sigma} \frac 1 s \D \mui{v}(s) +
\varepsiloni{u} \delta_0(\sigma), \quad \sigma \in
\borel{\rbb_+}.
   \end{align}
For this fix $u \in V$ and $i\in \nbb$. If
$\mu_u([0,i])=0$, then, according to our definitions,
we have $\lambdai{v}=0$ for all $v \in \dzi{u}$,
$\mui{u}=\delta_0$ and $\varepsiloni{u}=1$, which
means that the equality \eqref{wz1J} holds. Consider
now the case of $\mu_u([0,i])>0$. It follows from
\eqref{muu+} that
   \begin{align}   \label{muu}
\mu_u(\sigma \cap [0,i]) = \sum_{v\in \dzi{u}}
|\lambda_v|^2 \int_{\sigma \cap [0,i]} \frac 1 s \D
\mu_v(s) + \varepsilon_u \delta_0(\sigma), \quad
\sigma \in \borel{\rbb_+}.
   \end{align}
If $v \in \dzi{u}$ (equivalently:\ $u=\pa{v}$), then
by \eqref{wzd1} and \eqref{wzd2} we have
   \begin{align} \label{muu2}
   \begin{aligned}
\frac{|\lambda_v|^2}{\mu_u([0,i])} \int_{\sigma \cap
[0,i]} \frac 1 s \D \mu_v(s) & =
   \begin{cases}
   |\lambdai{v}|^2 \int_{\sigma} \frac 1 s \D
\mui{v}(s) & \text{ if } \mu_v([0,i]) > 0,
   \\[1ex]
   0 & \text{ if } \mu_v([0,i])=0,
   \end{cases}
   \\
   & = |\lambdai{v}|^2 \int_{\sigma} \frac 1 s \D
\mui{v}(s),
    \end{aligned}
   \end{align}
where the last equality holds because
$\lambdai{v}=0$ whenever $\mu_v([0,i])=0$.
Dividing both sides of \eqref{muu} by
$\mu_u([0,i])$ and using \eqref{muu2}, we obtain
\eqref{wz1J}.

Let $S_{\lambdabi}$ be the weighted shift on $\tcal$
with weights $\lambdabi$. Since, by \eqref{wzd2},
$\supp \mui{u} \subseteq [0,i]$ for every $u \in V$,
we infer from \eqref{wz1J} and Lemma~ \ref{lem3}\,(v),
applied to the triplet $(\lambdabi, \{\mui{v}\}_{v \in
V}, \{\varepsiloni{v}\}_{v \in V})$, that
$S_{\lambdabi}\in \ogr{\ell^2(V)}$. In turn,
\eqref{wz1J} and Lemma \ref{lem3}\,(iv) (applied to
the same triplet) imply that for every $u \in V$,
$\{\|S_{\lambdabi}^n e_u\|^2\}_{n=0}^\infty$ is a
Stieltjes moment sequence (with a representing measure
$\mui{u}$). Hence, by Theorem \ref{charsub}, the
operator $S_{\lambdabi}$ is subnormal.

Since $\mu_u$, $u \in V$, are Borel probability
measures on $\rbb_+$, we have
   \begin{align}  \label{lim1}
\lim_{i \to \infty} \mu_u([0,i]) = 1, \quad u \in V.
   \end{align}
Hence, for every $u \in V$ there exists a positive
integer $\kappa_{u}$ such that
   \begin{align}   \label{as1}
\mu_{u}([0,i]) > 0,\quad i \in \nbb, \, i \Ge
\kappa_{u}.
   \end{align}
Note that
   \begin{align} \label{wzj}
\lim_{i \to \infty} \lambdai{v} = \lambda_v, \quad v
\in V^\circ.
   \end{align}
Indeed, if $v \in V^\circ$, then \eqref{wzd1} and
\eqref{as1} yield $\lambdai{v}=\lambda_v
\sqrt{\frac{\mu_v([0,i])} {\mu_{\pa{v}}([0,i])}}$
for all integers $i \Ge \kappa_{\pa{v}}$. This,
combined with \eqref{lim1}, gives \eqref{wzj}. By
\eqref{wzd2}, \eqref{as1}, \eqref{wz1J} and Lemma
\ref{lem3}\,(iv), applied to $S_{\lambdabi}$, we
have
   \begin{align*}
\|S_{\lambdabi}^n e_u\|^2 = \int_0^\infty s^n \D
\mui{u}(s) = \frac{1}{\mu_u([0,i])} \int_{[0,i]} s^n
\D \mu_u(s), \quad n \in \zbb_+,\, i \Ge \kappa_{u},
\, u \in V.
   \end{align*}
This, together with \eqref{lim1} and Lemma
\ref{lem3}\,(iv), now applied to $\slam$, implies
that
   \begin{align}    \label{limsti}
\lim_{i \to \infty} \|S_{\lambdabi}^n e_u\|^2 =
\int_0^\infty s^n \D \mu_u(s) = \|\slam^n e_u\|^2,
\quad n \in \zbb_+,\, u \in V.
   \end{align}
It follows from \eqref{wzj}, \eqref{limsti} and
Proposition \ref{potegi} that \eqref{slim+} holds.
According to Proposition \ref{bas}\,(iv), $\escr$ is a
core of $\slam$. Hence $\lin \bigcup_{n=0}^\infty
\slam^n(\escr)$ is a core of $\slam$ as well. Applying
\eqref{slim+} and Theorem \ref{tw1} to the operators
$\{S_{\lambdabi}\}_{i=1}^\infty$ and $\slam$ with
$\xx:=\{e_u\colon u \in V\}$ completes the proof of
Theorem \ref{main}.
   \end{proof}
   \begin{rem}
In the proof of Theorem \ref{main} we have used
Proposition \ref{potegi} which provides a general
criterion for the validity of the approximation
procedure \eqref{slim+}. However, if the
approximating triplets $(\lambdabi,
\{\mui{v}\}_{v \in V}, \{\varepsiloni{v}\}_{v \in
V})$, $i=1,2,3, \ldots$, are defined as in
\eqref{wzd1}, \eqref{wzd2} and \eqref{wzd3}, then
   \begin{align} \label{slim2}
\lim_{i\to \infty} S_{\lambdabi}^n e_u =
S_{\lambdab}^n e_u, \quad u \in V, \, n \in
\zbb_+.
   \end{align}
To prove this, we first show that for all $u \in V$
and $i \Ge \kappa_u$ (see \eqref{as1}),
   \begin{align} \label{liuup}
\lambdai{u \mid u^\prime} = \lambda_{u \mid
u^\prime} \; \sqrt{\frac{\mu_{u^\prime}
([0,i])}{\mu_u([0,i])}}, \quad u^\prime \in
\dzin{n}{u}, \, n \in \zbb_+.
   \end{align}
Indeed, if $n=0$, then \eqref{liuup} holds.
Suppose that $n\Ge 1$. If
$\mu_{\paa(u^\prime)}([0,i])=0$, then $n\Ge 2$
and, by \eqref{wzd1}, $\lambdai{u^\prime} = 0$,
which implies that $\lambdai{u \mid u^\prime}=0$.
Since $\mu_{\paa(u^\prime)}([0,i])=0$, we deduce
from \eqref{muu+} (applied to $u=\paa(u^\prime)$)
that either $\lambda_{u^\prime}=0$, or
$\mu_{u^\prime} ([0,i]) = 0$. In both cases, the
right-hand side of \eqref{liuup} vanishes, and so
\eqref{liuup} holds. In turn, if
$\mu_{\paa(u^\prime)}([0,i]) > 0$, then we can
define
   \begin{align*}
j_0 = \min \Big\{j \in \{1, \ldots, n\}\colon
\mu_{\paa^k(u^\prime)}([0,i]) > 0 \text{ for all
} k=1, \ldots, j\Big\}.
   \end{align*}
Clearly, $1 \Le j_0 \Le n$. First, we consider
the case where $j_0 < n$. Since, by \eqref{as1},
$\mu_u([0,i])> 0$, we must have $j_0 \Le n-2$.
Thus $\mu_{\paa^{j_0+1}(u^\prime)}([0,i]) = 0$,
which together with \eqref{luv} and \eqref{wzd1}
implies that the left-hand side of \eqref{liuup}
vanishes. Since
$\mu_{\paa^{j_0+1}(u^\prime)}([0,i]) = 0$ and
$\mu_{\paa^{j_0}(u^\prime)}([0,i]) > 0$, we
deduce from \eqref{muu+} (applied to
$u=\paa^{j_0+1}(u^\prime)$) that
$\lambda_{\paa^{j_0}(u^\prime)}=0$, and so the
right-hand side of \eqref{liuup} vanishes. This
means that \eqref{liuup} is again valid. Finally,
if $j_0=n$, then by \eqref{wzd1} we have
   \begin{align*}
\lambdai{u \mid u^\prime} = \prod_{j=0}^{n-1}
\lambda_{\paa^j(u^\prime)}
\sqrt{\frac{\mu_{\paa^j(u^\prime)([0,i])}}
{\mu_{\paa^{j+1}(u^\prime)([0,i])}}} = \lambda_{u
\mid u^\prime} \; \sqrt{\frac{\mu_{u^\prime}
([0,i])}{\mu_u([0,i])}},
   \end{align*}
which completes the proof of \eqref{liuup}. Now we
show that
   \begin{align} \label{ils}
\lim_{i \to \infty} \is{S_{\lambdab}^n
e_u}{S_{\lambdabi}^n e_u} = \|\slam^n e_u\|^2,
\quad u \in V, \, n \in \zbb_+.
   \end{align}
Indeed, it follows from Lemma \ref{lem4}(ii) and
\eqref{liuup} that
   \begin{multline*}
\is{S_{\lambdab}^n e_u}{S_{\lambdabi}^n e_u} =
\sum_{u^\prime \in \dzin{n}{u}} \lambda_{u\mid
u^\prime} \overline{\lambdai{u \mid u^\prime}}
   \\
= \frac{1}{\sqrt{\mu_u([0,i])}} \sum_{u^\prime
\in \dzin{n}{u}} |\lambda_{u\mid u^\prime}|^2
\sqrt{\mu_{u^\prime}([0,i])}, \quad u \in V, \, n
\in \zbb_+, \, i \Ge \kappa_u.
   \end{multline*}
By applying Lebesgue's monotone convergence theorem
for series, \eqref{lim1} and Lem\-ma \ref{lem4}(iii),
we obtain \eqref{ils}. Since
   \begin{align*}
\|S_{\lambdab}^n e_u - S_{\lambdabi}^n e_u\|^2 =
\|S_{\lambdab}^n e_u\|^2 + \|S_{\lambdabi}^n
e_u\|^2 - 2 \, \mathrm{Re}\is{S_{\lambdab}^n
e_u}{S_{\lambdabi}^n e_u}
   \end{align*}
we infer \eqref{slim2} from \eqref{limsti} and
\eqref{ils}. Clearly \eqref{slim2} implies
\eqref{slim+}.
   \end{rem}
We conclude this section with a general criterion for
subnormality of weighted shifts on directed trees
written in terms of determinacy of Stieltjes moment
sequences.
   \begin{cor} \label{necessdet2}
Let $\slam$ be a weighted shift on a directed tree
$\tcal$ with weights $\lambdab=\{\lambda_v\}_{v \in
V^\circ}$ such that $\escr \subseteq \dzn{\slam}$.
Assume that $\{\|\slam^{n+1} e_u\|^2\}_{n=0}^\infty$
is a determinate Stieltjes moment sequence for every
$u \in V$. Then the following conditions are
equivalent{\em :}
   \begin{enumerate}
   \item[(i)] $\slam$ is subnormal,
   \item[(ii)] $\{\|\slam^{n} e_u\|^2\}_{n=0}^\infty$
is a Stieltjes moment sequence for every $u \in V$,
   \item[(iii)] there exist a system $\{\mu_u\}_{u
\in V}$ of Borel probability measures on $\rbb_+$
and a system $\{\varepsilon_u\}_{u \in V}$ of
nonnegative real numbers that satisfy
\eqref{muu+} for every $u \in V$.
   \end{enumerate}
   \end{cor}
   \begin{proof}
(i)$\Rightarrow$(ii) Use Proposition \ref{necess}.

(ii)$\Rightarrow$(iii) Employ Lemma \ref{2necess+}.

(iii)$\Rightarrow$(i) Apply Theorem \ref{main}.
   \end{proof}
Regarding Corollary \ref{necessdet2}, note that
by Proposition \ref{necess}, Lemma
\ref{lem3}\,(iv) and \eqref{st+1} each of the
conditions (i), (ii) and (iii) implies that
$\{\|\slam^{n+1} e_u\|^2\}_{n=0}^\infty$ is a
Stieltjes moment sequence for every $u \in V$.
   \subsection{Nonzero weights}
   As pointed out in \cite[Proposition 5.1.1]{j-j-s}
bounded hyponormal weighted shifts on directed trees
with nonzero weights are always injective. It turns
out that the same conclusion can be derived in the
unbounded case (with almost the same proof). Recall
that a densely defined operator $S$ in $\hh$ is said
to be {\em hyponormal} if $\dz{S} \subseteq \dz{S^*}$
and $\|S^*f\| \Le \|Sf\|$ for all $f \in \dz S$. It is
well-known that subnormal operators are hyponormal
(but not conversely) and that hyponormal operators are
closable and their closures are hyponormal. We refer
the reader to \cite{ot-sch,jj1,jj2,jj3,sto} for
elements of the theory of unbounded hyponormal
operators.
   \begin{pro} \label{hypcor}
Let $\tcal$ be a directed tree with $V^\circ \neq
\varnothing$. If $\slam$ is a hyponormal weighted
shift on $\tcal$ whose all weights are nonzero, then
$\tcal$ is leafless. In particular, $\slam$ is
injective and $V$ is infinite and countable.
   \end{pro}
   \begin{proof}
Suppose that, contrary to our claim, $\dzi u =
\varnothing$ for some $u \in V$. We deduce from
Proposition \ref{przem} and $V^\circ \neq \varnothing$
that $u \in V^\circ$. Hence, by assertions (ii), (iii)
and (vi) of Proposition \ref{bas}, we have
   \begin{align*}
|\lambda_u|^2 \overset{ \eqref{sl*}}= \|\slam^*e_u\|^2
\Le \|\slam e_u\|^2 \overset{\eqref{eu}}= \sum_{v \in
\dzi u} |\lambda_v|^2 = 0,
   \end{align*}
which is a contradiction. Since each leafless directed
tree is infinite, we deduce from assertions (vii) and
(viii) of Proposition \ref{bas} that $\slam$ is
injective and $V$ is infinite and countable. This
completes the proof.
   \end{proof}
The sufficient condition for subnormality of
weighted shifts on directed trees stated in
Theorem \ref{main} takes the simplified form for
weighted shifts with nonzero weights. Indeed, if
a weighted shift $\slam$ on $\tcal$ with nonzero
weights satisfies the assumptions of Theorem
\ref{main}, then, by assertions (ii) and (iii) of
Lemma \ref{lem1}, $\varepsilon_v=0$ for every $v
\in V^\circ$.
   \begin{cor}
   Let $\slam$ be a weighted shift on a directed tree
$\tcal$ with nonzero weights
$\lambdab=\{\lambda_v\}_{v \in V^\circ}$ such that
$\escr \subseteq \dzn{\slam}$. Then $\slam$ is
subnormal provided that one of the following two
conditions holds\/{\em :}
   \begin{enumerate}
   \item[(i)] $\tcal$ is rootless and
there exists a system $\{\mu_v\}_{v \in V}$ of Borel
probability measures on $\rbb_+$ which satisfies the
following equality for every $u \in V$,
   \begin{align} \label{wz1+}
\mu_u(\sigma) = \sum_{v\in \dzi{u}} |\lambda_v|^2
\int_{\sigma} \frac 1 s \D \mu_v(s), \quad \sigma \in
\borel{\rbb_+},
   \end{align}
   \item[(ii)] $\tcal$ has a root and
there exist $\varepsilon \in \rbb_+$ and a system
$\{\mu_v\}_{v \in V}$ of Borel probability measures on
$\rbb_+$ which satisfy \eqref{wz1+} for every $u \in
V^\circ$, and
   \begin{align*}
\mu_{\koo}(\sigma) = \sum_{v\in \dzi{\koo}}
|\lambda_v|^2 \int_{\sigma} \frac 1 s \D \mu_v(s) +
\varepsilon \delta_0(\sigma), \quad \sigma \in
\borel{\rbb_+}.
   \end{align*}
   \end{enumerate}
   \end{cor}
   \subsection{\label{cfs}Quasi-analytic vectors}
Let $S$ be an operator in a complex Hilbert space
$\hh$. We say that a vector $f\in\dzn{S}$ is a {\em
quasi-analytic} vector of $S$ if
   \begin{align*}
\sum_{n=1}^\infty \frac{1}{\|S^n
f\|^{\nicefrac{1}{n}}} = \infty \quad
\text{(convention: $\frac{1}{0}=\infty$)}.
   \end{align*}
Denote by $\quasi{S}$ the set of all quasi-analytic
vectors. Note that (cf.\ \cite[Section 9]{StSz1})
   \begin{align} \label{quasiinv}
   S(\quasi{S}) \subseteq \quasi{S}.
   \end{align}
In general, $\quasi{S}$ is not a linear subspace of
$\hh$ even if $S$ is essentially selfadjoint (see
\cite{ru2}; see also \cite{ru1} for related matter).

We now show that the converse of the implication in
Proposition \ref{necess} holds for weighted shifts on
directed trees having sufficiently many quasi-analytic
vectors, and that within this class of operators
subnormality is completely characterized by the
existence of a consistent system of probability
measures.
   \begin{thm} \label{main-0}
   Let $\slam$ be a weighted shift on a directed tree
$\tcal$ with weights $\lambdab=\{\lambda_v\}_{v \in
V^\circ}$ such that $\escr \subseteq \quasi{\slam}$.
Then the following conditions are equivalent{\em :}
   \begin{enumerate}
   \item[(i)] $\slam$ is subnormal,
   \item[(ii)] $\{\|\slam^n e_u\|^2\}_{n=0}^\infty$ is
a Stieltjes moment sequence for every $u \in V$,
   \item[(iii)] there exist a system $\{\mu_v\}_{v \in V}$ of Borel
probability measures on $\rbb_+$ and a system
$\{\varepsilon_v\}_{v \in V}$ of nonnegative real
numbers that satisfy \eqref{muu+} for every $u
\in V$.
   \end{enumerate}
   \end{thm}
   \begin{proof}
(i)$\Rightarrow$(ii) Apply Proposition \ref{necess}.

(ii)$\Rightarrow$(iii) Fix $u \in V$ and set $t_n =
\|\slam^{n+1} e_u\|^2$ for $n \in \zbb_+$. By
\eqref{st+1}, the sequence $\{t_n\}_{n=0}^\infty$ is a
Stieltjes moment sequence. Since $e_u \in
\quasi{\slam}$, we infer from \eqref{quasiinv} that
$\slam e_u \in \quasi{\slam}$, or equivalently that
$\sum_{n=1}^\infty t_n^{-\nicefrac{1}{2n}} = \infty$.
Hence, by the Carleman criterion for determinacy of
Stieltjes moment sequences\footnote{\;In fact, one can
prove that a Stieltjes moment sequence
$\{t_n\}_{n=0}^\infty$ for which $\sum_{n=1}^\infty
t_n^{-\nicefrac{1}{2n}} = \infty$ is determinate as a
Hamburger moment sequence, which means that there
exists only one positive Borel measure on $\rbb$ which
represents the sequence $\{t_n\}_{n=0}^\infty$ (cf.\
\cite[Corollary 4.5]{sim}).} (cf.\ \cite[Theorem
1.11]{sh-tam}), the Stieltjes moment sequence
$\{t_n\}_{n=0}^\infty = \{\|\slam^{n+1}
e_u\|^2\}_{n=0}^\infty$ is determinate. Now applying
Lemma \ref{2necess+} yields (iii).

(iii)$\Rightarrow$(i) Employ Theorem \ref{main}.
   \end{proof}

Using \cite[Theorem 7]{StSz1}, one can prove a version
of Theorem \ref{main-0} in which the class of
quasi-analytic vectors is replaced by the class of
analytic ones. Since the former class is
larger\footnote{\;In general, the class of analytic
vectors of an operator $S$ is essentially smaller than
the class of quasi-analytic vectors of $S$ even for
essentially selfadjoint operators $S$ (cf.\
\cite{ru1}).} than the latter, we see that
``analytic'' version of Theorem \ref{main-0} is weaker
than Theorem \ref{main-0} itself. To the best of our
knowledge, Theorem \ref{main-0} is the first result of
this kind; it shows that the unbounded version of
Lambert's characterization of subnormality happens to
be true for operators that have sufficiently many
quasi-analytic vectors.

The following result, which is an immediate
consequence of Theorem \ref{main-0}, provides a
new characterization of subnormality of bounded
weighted shifts on directed trees written in
terms of consistent systems of probability
measures. It may be thought of as a complement to
Theorem \ref{charsub}.
   \begin{cor}
Let $\slam \in \ogr{\ell^2(V)}$ be a weighted
shift on a directed tree $\tcal$ with weights
$\lambdab=\{\lambda_v\}_{v \in V^\circ}$. Then
$\slam$ is subnormal if and only if there exist a
system $\{\mu_v\}_{v \in V}$ of Borel probability
measures on $\rbb_+$ and a system
$\{\varepsilon_v\}_{v \in V}$ of nonnegative real
numbers that satisfy \eqref{muu+} for every $u
\in V$.
   \end{cor}
   \subsection{Subnormality via subtrees}
Let $\slam$ be a weighted shift on a directed tree
$\tcal$ with weights $\lambdab=\{\lambda_v\}_{v\in
V^\circ}$. Note that if $u \in V$, then the space
$\ell^2(\des{u})$ (which is regarded as a closed
linear subspace of $\ell^2(V)$) is invariant for
$\slam$, i.e.,
   \begin{align} \label{ilb}
\slam\big(\dz{\slam} \cap \ell^2(\des{u})\big)
\subseteq \ell^2(\des{u}).
   \end{align}
(For this, apply \eqref{lamtauf} and the inclusion
$\paa\big(V\setminus \big(\des{u} \cup
\Ko{\tcal}\big)\big) \subseteq V\setminus \des{u}$.)
Denote by $\slam|_{\ell^2(\des{u})}$ the operator in
$\ell^2(\des{u})$ given by
$\dz{\slam|_{\ell^2(\des{u})}} = \dz{\slam} \cap
\ell^2(\des{u})$ and $\slam|_{\ell^2(\des{u})}f =
\slam f$ for $f \in \dz{\slam|_{\ell^2(\des{u})}}$. It
is easily seen that $\slam|_{\ell^2(\des{u})}$
coincides with the weighted shift on the directed tree
$(\des{u}, (\des{u}\times \des{u}) \cap E)$ with
weights $\{\lambda_v\}_{v \in \des{u}\setminus \{u\}}$
(see \cite[Proposition 2.1.8]{j-j-s} for more details
on this and related subtrees).

Proposition \ref{subtree} below shows that the study
of subnormality of weighted shifts on rootless
directed trees can be reduced in a sense to the case
of directed trees with root. Unfortunately, our
criteria for subnormality of weighted shifts on
directed trees are not applicable in this context.
Fortunately, we can employ the inductive limit
approach to subnormality provided by Proposition
\ref{tw1+1}.
   \begin{pro}\label{subtree} Let $\slam$ be a weighted
shift on a rootless directed tree $\tcal$ with weights
$\lambdab=\{\lambda_v\}_{v\in V^\circ}$. Suppose that
$\escr \subseteq \dzn{\slam}$. If $\varOmega$ is a
subset of $V$ such that $V=\bigcup_{\omega\in
\varOmega} \des{\omega}$, then the following
conditions are equivalent{\em :}
   \begin{enumerate}
   \item[(i)] $\slam$ is subnormal,
   \item[(ii)] for every $\omega\in \varOmega$,
$\slam|_{\ell^2(\des{\omega})}$ is subnormal as an
operator acting in $\ell^2(\des{\omega})$.
   \end{enumerate}
   \end{pro}
   \begin{proof}
(ii)$\Rightarrow$(i) Using an induction argument and
\eqref{ilb} one can show that $\slam^n e_v \in
\ell^2(\des{v}) \subseteq \ell^2(\des{u})$ for all
$n\in \zbb_+$, $v \in \des{u}$ and $u \in V$. Hence
   \begin{align*}
\xx_\omega := \lin\big\{e_v\colon v \in
\des{\omega}\big\} \subseteq \dzn{\slam} \text{ and }
\slam^n(\xx_\omega) \subseteq \ell^2(\des{\omega})
   \end{align*}
for all $\omega \in \varOmega$ and $n\in \zbb_+$. It
follows from \cite[Proposition 2.1.4]{j-j-s} and the
equality $V=\bigcup_{\omega\in \varOmega}
\des{\omega}$ that for each pair $(\omega_1,\omega_2)
\in \varOmega \times \varOmega$, there exists $\omega
\in \varOmega$ such that $\des{\omega_1} \cup
\des{\omega_2} \subseteq \des{\omega}$, and thus
$\{\xx_\omega\}_{\omega \in \varOmega}$ is an upward
directed family of subsets of $\ell^2(V)$. By applying
Proposition \ref{bas}(iv) and Proposition \ref{tw1+1}
to $S=\slam$ and $\hh_\omega= \ell^2(\des{\omega})$,
we get (i).

The reverse implication (i)$\Rightarrow$(ii) is
obvious because $\xx_\omega \subseteq
\dz{\slam|_{\ell^2(\des{\omega})}}$.
   \end{proof}
It follows from \cite[Proposition 2.1.6]{j-j-s} that
if $\tcal$ is a rootless directed tree, then
$V=\bigcup_{k=1}^\infty \des{\paa^k(u)}$ for every $u
\in V$, and so the set $\varOmega$ in Proposition
\ref{subtree} may always be chosen to be countable and
infinite.
   \section{Subnormality on Assorted Directed Trees}
   \subsection{\label{cws}Classical weighted shifts}
By a {\em classical weighted shift} we mean either a
unilateral weighted shift $S$ in $\ell^2$ or a
bilateral weighted shift $S$ in $\ell^2(\zbb)$, i.e.,
$S=VD$, where, in the unilateral case, $V$ is the
unilateral isometric shift on $\ell^2$ of multiplicity
$1$ and $D$ is a diagonal operator in $\ell^2$ with
diagonal elements $\{\lambda_n\}_{n=0}^\infty$; in the
bilateral case, $V$ is the bilateral unitary shift on
$\ell^2(\zbb)$ of multiplicity $1$ and $D$ is a
diagonal operator in $\ell^2(\zbb)$ with diagonal
elements $\{\lambda_n\}_{n=-\infty}^\infty$. In view
of \cite[equality (1.7)]{ml}, $S$ is a unique closed
linear operator in $\ell^2$ (respectively:\
$\ell^2(\zbb)$) such that the linear span of the
standard orthonormal basis $\{e_n\}_{n=0}^\infty$ of
$\ell^2$ (respectively:\ $\{e_n\}_{n=-\infty}^\infty$
of $\ell^2(\zbb)$) is a core of $S$ and
   \begin{align} \label{notold}
S e_n = \lambda_n e_{n+1}, \quad n\in \zbb_+ \;\;
(\textrm{respectively:\ } n \in \zbb).
   \end{align}
This fact, combined with parts (ii), (iii) and (iv) of
Proposition \ref{bas}, implies that a unilateral
(respectively:\ a bilateral) classical weighted shift
is a weighted shift on the directed tree $(\zbb_+,
\{(n,n+1)\colon n \in \zbb_+\})$ (respectively:\
$(\zbb, \{(n,n+1)\colon n \in \zbb\})$) with weights
$\{\lambda_{n-1}\}_{n=1}^\infty$ (respectively:\
$\{\lambda_{n-1}\}_{n=-\infty}^\infty$). From now on
we enumerate weights of a classical weighted shift in
accordance with our notation relevant to these two
particular trees. This means that \eqref{notold} takes
now the form
   \begin{align} \label{notnew}
\slam e_n = \lambda_{n+1} e_{n+1}, \quad n\in \zbb_+
\;\; (\textrm{respectively:\ } n \in \zbb),
   \end{align}
where $\lambdab=\{\lambda_{n}\}_{n=1}^\infty$
(respectively:\
$\lambdab=\{\lambda_{n}\}_{n=-\infty}^\infty$).

Using our approach, we can derive the
Berger-Gellar-Wallen criterion for subnormality of
injective unilateral classical weighted shifts (see
\cite{g-w,hal2} for the bound\-ed case and
\cite[Theorem 4]{StSz1} for the unbounded one).
   \begin{thm} \label{b-g-w}
If $\slam$ is a unilateral classical weighted shift
with nonzero weights $\lambdab =
\{\lambda_n\}_{n=1}^\infty$ $($with notation as in
\eqref{notnew}$)$, then the following three conditions
are equivalent\/{\em :}
   \begin{enumerate}
   \item[(i)] $\slam$ is subnormal,
   \item[(ii)]  $\{1, |\lambda_1|^2, |\lambda_1
\lambda_2|^2, |\lambda_1 \lambda_2 \lambda_3|^2,
\ldots\}$ is a Stieltjes moment sequence,
   \item[(iii)]
$\{\|\slam^n e_k\|^2\}_{n=0}^\infty$ is a Stieltjes
moment sequence for all $k \in \zbb_+$.
   \end{enumerate}
   \end{thm}
   \begin{proof}
First note that $\escr \subseteq \dzn{\slam}$.

(i)$\Rightarrow$(iii) Employ Proposition \ref{necess}.

(iii)$\Rightarrow$(ii) This is evident, because the
sequence $\{1, |\lambda_1|^2, |\lambda_1 \lambda_2|^2,
|\lambda_1 \lambda_2 \lambda_3|^2, \ldots\}$ coincides
with $\{\|\slam^n e_0\|^2\}_{n=0}^\infty$.

(ii)$\Rightarrow$(i) Let $\mu$ be a representing measure of
the Stieltjes moment sequence $\{\|\slam^n
e_0\|^2\}_{n=0}^\infty$ (which in general may not be
determinate, cf.\ \cite{sz3}). Define the sequence
$\{\mu_n\}_{n=0}^\infty$ of Borel probability measures on
$\rbb_+$ by
   \begin{align*}
\mu_n(\sigma) = \frac{1}{\|\slam^n e_0\|^2}
\int_{\sigma} s^n \D \mu(s), \quad \sigma \in
\borel{\rbb_+}, \, n \in \zbb_+.
   \end{align*}
It is then clear that
   \begin{align*}
\mu_0(\sigma) &=|\lambda_{1}|^2 \int_\sigma
\frac{1}{s} \D\mu_{1}(s) + \mu(\{0\}) \delta_0
(\sigma), \quad \sigma \in \borel{\rbb_+},
   \\
\mu_n(\sigma) &= |\lambda_{n+1}|^2 \int_\sigma
\frac{1}{s} \D\mu_{n+1}(s), \quad \sigma \in
\borel{\rbb_+},\, n \Ge 1,
   \end{align*}
which means that the systems $\{\mu_n\}_{n=0}^\infty$
and $\{\varepsilon_n\}_{n=0}^\infty := \{\mu(\{0\}),
0, 0, \ldots\}$ satisfy the assumptions of Theorem
\ref{main}. This completes the proof.
   \end{proof}
Before formulating the next theorem, we recall that a
two-sided sequence $\{t_n\}_{n=-\infty}^\infty$ of
real numbers is said to be a {\em two-sided Stieltjes
moment sequence} if there exists a positive Borel
measure $\mu$ on $(0,\infty)$ such that
   \begin{align*}
t_{n}=\int_{(0,\infty)} s^n \D\mu(s),\quad n \in \zbb;
   \end{align*}
$\mu$ is called a {\em representing measure} of
$\{t_n\}_{n=-\infty}^\infty$. It follows from
\cite[page 202]{ber} (see also \cite[Theorem
6.3]{j-t-w}) that
   \begin{align} \label{char2sid}
   \begin{minipage}{29em}
$\{t_n\}_{n=-\infty}^\infty \subseteq \rbb$ is a
two-sided Stieltjes moment sequence if and only if
$\{t_{n-k}\}_{n=0}^\infty$ is a Stieltjes moment
sequence for every $k \in \zbb_+$.
   \end{minipage}
   \end{align}
Now we show how to deduce an analogue of the
Berger-Gellar-Wallen criterion for subnormality of
injective bilateral classical weighted shifts from our
results (see \cite[Theorem II.6.12]{con2} for the
bounded case and \cite[Theorem 5]{StSz1} for the
unbounded~ one).
   \begin{thm} \label{b-g-w-2}
If $\slam$ is a bilateral classical weighted shift
with nonzero weights $\lambdab=\{\lambda_n\}_{n \in
\zbb}$ $($with notation as in \eqref{notnew}$)$, then
the following four conditions are equivalent\/{\em :}
   \begin{enumerate}
   \item[(i)] $\slam$ is subnormal,
   \item[(ii)] the two-sided sequence $\{t_n\}_{n=-\infty}^\infty$
defined by
   \begin{align*}
t_n =
   \begin{cases}
|\lambda_1 \cdots \lambda_{n}|^2 & \text{ for } n \Ge
1,
   \\
1 & \text{ for } n=0,
   \\
|\lambda_{n+1} \cdots \lambda_{0}|^{-2} & \text{ for }
n \Le -1,
   \end{cases}
   \end{align*}
is a two-sided Stieltjes moment sequence,
   \item[(iii)]
$\{\|\slam^n e_{-k}\|^2\}_{n=0}^\infty$ is a Stieltjes
moment sequence for infinitely many nonnegative
integers $k$,
   \item[(iv)]
$\{\|\slam^n e_k\|^2\}_{n=0}^\infty$ is a Stieltjes
moment sequence for all $k \in \zbb$.
   \end{enumerate}
   \end{thm}
   \begin{proof}
First note that $\escr \subseteq \dzn{\slam}$.

(i)$\Rightarrow$(iv) Employ Proposition \ref{necess}.

(iv)$\Rightarrow$(iii) Evident.

(iii)$\Rightarrow$(iv) Apply Lemma \ref{charsub-1}.

(iv)$\Rightarrow$(ii) Since $t_{n-k} = t_{-k}
\|\slam^n e_{-k}\|^2$ for all $n \in \zbb$ and $k\in
\zbb_+$, we can apply the criterion \eqref{char2sid}.

(ii)$\Rightarrow$(i) Let $\mu$ be a representing
measure of $\{t_n\}_{n=-\infty}^\infty$. Define the
two-sided sequence $\{\mu_n\}_{n=-\infty}^\infty$ of
Borel probability measures on $\rbb_+$ by (note that
$\mu(\{0\})=0$)
   \begin{align*}
\mu_n(\sigma) = \frac{1}{\|\slam^n e_0\|^2}
\int_{\sigma} s^n \D \mu(s), \quad \sigma \in
\borel{\rbb_+}, \, n \in \zbb.
   \end{align*}
We easily verify that
   \begin{align*}
\mu_n(\sigma) &= |\lambda_{n+1}|^2 \int_\sigma
\frac{1}{s} \D\mu_{n+1}(s), \quad \sigma \in
\borel{\rbb_+},\, n \in \zbb,
   \end{align*}
which means that the systems
$\{\mu_n\}_{n=-\infty}^\infty$ and
$\{\varepsilon_n\}_{n=-\infty}^\infty$ with
$\varepsilon_n\equiv 0$ satisfy the assumptions of
Theorem \ref{main}. This completes the proof.
   \end{proof}
It is worth mentioning that, in view of Theorems
\ref{b-g-w} and \ref{b-g-w-2}, the necessary
condition for subnormality of Hilbert space
operators that appeared in Proposition
\ref{necess-gen} (see also Proposition
\ref{necess}) turns out to be sufficient in the
case of injective classical weighted shifts. To
the best of our knowledge, the class of injective
classical weighted shifts seems to be the only
one for which this phenomenon occurs regardless
of whether or not the operators in question have
sufficiently many qusi-analytic vectors (see
\cite{StSz0} for more details; see also Sections
\ref{subs1} and \ref{cfs}).
   \subsection{\label{obv}One branching vertex}
Our next aim is to discuss subnormality of weighted
shifts with nonzero weights on leafless directed trees
that have only one branching vertex. Such directed
trees are one step more complicated than those
involved in the definitions of classical weighted
shifts (see Section \ref{cws}). By Proposition
\ref{hypcor}, there is no loss of generality in
assuming that $\card{V} = \aleph_0$. Infinite,
countable and leafless directed trees with one
branching vertex can be modelled as follows (see
Figure 1). Given $\eta,\kappa \in \zbb_+ \sqcup
\{\infty\}$ with $\eta \Ge 2$, we define the directed
tree $\tcal_{\eta,\kappa} = (V_{\eta,\kappa},
E_{\eta,\kappa})$ by
   \allowdisplaybreaks
   \begin{align*}
   \begin{aligned}
V_{\eta,\kappa} & = \big\{-k\colon k\in J_\kappa\big\}
\sqcup \{0\} \sqcup \big\{(i,j)\colon i\in J_\eta,\,
j\in \nbb\big\},
   \\
E_{\eta,\kappa} & = E_\kappa \sqcup
\big\{(0,(i,1))\colon i \in J_\eta\big\} \sqcup
\big\{((i,j),(i,j+1))\colon i\in J_\eta,\, j\in
\nbb\big\},
   \\
E_\kappa & = \big\{(-k,-k+1) \colon k\in
J_\kappa\big\},
   \end{aligned}
   \end{align*}
where $J_\iota := \{k \in \nbb\colon k\Le \iota\}$ for
$\iota \in \zbb_+ \sqcup \{\infty\}$.
   \vspace{1.5ex}
   \begin{center}
   \includegraphics[width=7cm]
{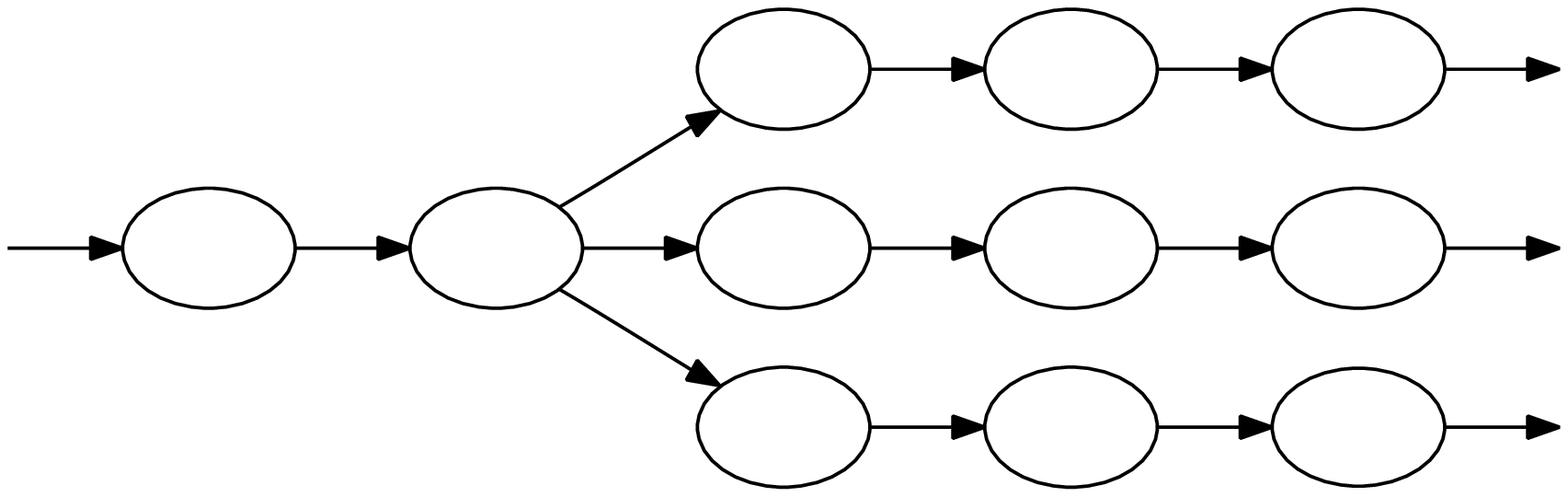}
   \\[1.5ex]
{\small {\sf Figure 1}}
   \end{center}
   \vspace{1ex}
   If $\kappa < \infty$, then the directed tree
$\tcal_{\eta,\kappa}$ has the root $-\kappa$. If
$\kappa=\infty$, then the directed tree
$\tcal_{\eta,\infty}$ is rootless. In all cases,
$0$ is the branching vertex of
$\tcal_{\eta,\kappa}$.

   We begin by proving criteria for subnormality of
weighted shifts on $\tcal_{\eta,\kappa}$ with nonzero
weights. Below, we adhere to the notation
$\lambda_{i,j}$ instead of a more formal expression
$\lambda_{(i,j)}$.
   \begin{thm}\label{omega2}
Let $\slam$ be a weighted shift on the directed tree
$\tcal_{\eta,\kappa}$ with nonzero weights $\lambdab =
\{\lambda_v\}_{v \in V_{\eta,\kappa}^\circ}$ such that
$e_0 \in \dzn{\slam}$. Suppose that there exists a
sequence $\{\mu_i\}_{i=1}^\eta$ of Borel probability
measures on $\rbb_+$ such that
   \begin{align} \label{zgod0}
\int_0^\infty s^n \D \mu_i(s) =
\Big|\prod_{j=2}^{n+1}\lambda_{i,j}\Big|^2, \quad n
\in \nbb, \; i \in J_\eta.
   \end{align}
Then $\slam$ is subnormal provided that one of the
following four conditions holds\/{\em :}
   \begin{enumerate}
   \item[(i)]  $\kappa=0$ and
   \begin{align}  \label{zgod}
\sum_{i=1}^\eta |\lambda_{i,1}|^2 \int_0^\infty \frac
1 s\, \D \mu_i(s) \Le 1,
   \end{align}
   \item[(ii)] $0 < \kappa < \infty$ and
   \begin{align} \label{zgod'}
\sum_{i=1}^\eta |\lambda_{i,1}|^2 \int_0^\infty \frac
1 s\, \D \mu_i(s) &= 1,
   \\
\Big|\prod_{j=0}^{l-1} \lambda_{-j}\Big|^2
\sum_{i=1}^\eta|\lambda_{i,1}|^2 \int_0^\infty \frac 1
{s^{l+1}} \D \mu_i(s) & = 1, \quad l \in J_{\kappa-1},
\label{widly1}
   \\
\Big|\prod_{j=0}^{\kappa-1}
\lambda_{-j}\Big|^2\sum_{i=1}^\eta|\lambda_{i,1}|^2
\int_0^\infty \frac 1 {s^{\kappa+1}} \D \mu_i(s) & \Le
1, \label{widly1'}
   \end{align}
   \item[(iii)] $0 < \kappa < \infty$ and there exists
a Borel probability measure $\nu$ on $\rbb_+$ such
that
   \begin{align} \label{prob}
\int_0^\infty s^n \D \nu(s) & =
\Big|\prod_{j=\kappa-n}^{\kappa-1}\lambda_{-j}\Big|^2,
\quad n \in J_\kappa,
   \\     \label{prob'}
\int_\sigma s^\kappa \D \nu(s) & =
\Big|\prod_{j=0}^{\kappa-1} \lambda_{-j}\Big|^2 \;
\sum_{i=1}^\eta |\lambda_{i,1}|^2 \int_\sigma
\frac{1}{s} \D \mu_i(s), \quad \sigma \in
\borel{\rbb_+},
   \end{align}
   \item[(iv)] $\kappa=\infty$ and equalities \eqref{zgod'}
and \eqref{widly1} are satisfied.
   \end{enumerate}
   \end{thm}
   \begin{proof} Note that the assumption $e_0\in
\dzn{\slam}$ implies that
$\mathscr{E}_{V_{\eta,\kappa}} \subseteq \dzn{\slam}$.

(i) Define the system of Borel probability measures
$\{\mu_v\}_{v\in V_{\eta,0}}$ on $\rbb_+$ and the
system $\{\varepsilon_v\}_{v\in V_{\eta,0}}$ of
nonnegative real numbers by
   \begin{align*}
\mu_{0}(\sigma) & = \sum_{i=1}^{\eta} |\lambda_{i,1}|^2
\int_\sigma \frac 1 s \D \mu_i(s) + \varepsilon_0
\delta_0(\sigma), \quad \sigma \in \borel{\rbb_+},
   \\
\varepsilon_0 & = 1 - \sum_{i=1}^\eta
|\lambda_{i,1}|^2 \int_0^\infty \frac 1 s\, \D
\mu_i(s),
   \end{align*}
and
   \begin{align} \label{kap0}
\mu_{i,n}(\sigma) & = \frac{1}{\|\slam^{n-1}
e_{i,1}\|^2} \int_{\sigma} s^{n-1} \D \mu_i(s), \quad
\sigma \in \borel{\rbb_+}, \, i \in J_\eta, \, n \in
\nbb,
   \\
\varepsilon_{i,n} & = 0, \quad i \in J_\eta, \, n \in
\nbb. \notag
   \end{align}
(We write $\mu_{i,j}$ and $\varepsilon_{i,j}$ instead
of the more formal expressions $\mu_{(i,j)}$ and
$\varepsilon_{(i,j)}$.) Clearly $\mu_{i,1}=\mu_i$ for
all $i \in J_\eta$. Using \eqref{zgod0} and
\eqref{zgod}, we verify that the systems
$\{\mu_v\}_{v\in V_{\eta,0}}$ and
$\{\varepsilon_v\}_{v\in V_{\eta,0}}$ are well-defined
and satisfy the assumptions of Theorem \ref{main}.
Hence $\slam$ is subnormal.

(ii) Define the systems $\{\mu_v\}_{v\in
V_{\eta,\kappa}}$ and $\{\varepsilon_v\}_{v\in
V_{\eta,\kappa}}$ by \eqref{kap0} and
   \begin{align}\label{literki1}
\mu_{0}(\sigma) & = \sum_{i = 1}^{\eta}
|\lambda_{i,1}|^2 \int_\sigma \frac 1 s \D \mu_i(s),
\quad \sigma \in \borel{\rbb_+},
   \\   \label{literki2}
   \mu_{-l} (\sigma) & = \Big|\prod_{j=0}^{l-1}
\lambda_{-j}\Big|^2 \sum_{i=1}^{\eta} |\lambda_{i,1}|^2
\int_{\sigma} \frac 1 {s^{l+1}}\, \D \mu_i(s), \quad
\sigma \in \borel{\rbb_+}, \, l\in J_{\kappa-1},
   \\ \label{literki3}
   \mu_{-\kappa} (\sigma) & =
\Big|\prod_{j=0}^{\kappa-1} \lambda_{-j}\Big|^2
\sum_{i=1}^{\eta} |\lambda_{i,1}|^2 \int_{\sigma} \frac
1 {s^{\kappa+1}}\, \D \mu_i(s) + \varepsilon_{-\kappa}
\delta_0(\sigma), \hspace{0.8ex} \sigma \in
\borel{\rbb_+}, \
   \\  \label{literki4}
\varepsilon_v & =
   \begin{cases}
0 & \text{ if } v\in V_{\eta,\kappa}^\circ,
   \\
   1 - \Big|\prod_{j=0}^{\kappa-1}
\lambda_{-j}\Big|^2\sum_{i=1}^\eta|\lambda_{i,1}|^2
\int_0^\infty \frac 1 {s^{\kappa+1}} \D \mu_i(s) &
\text{ if } v = - \kappa.
   \end{cases}
   \end{align}
Applying \eqref{zgod0}, \eqref{zgod'}, \eqref{widly1}
and \eqref{widly1'}, we check that the systems
$\{\mu_v\}_{v\in V_{\eta,\kappa}}$ and
$\{\varepsilon_v\}_{v\in V_{\eta,\kappa}}$ are
well-defined and satisfy the assumptions of Theorem
\ref{main}. Therefore $\slam$ is subnormal.

(iii) First note that $\|\slam^n e_{-\kappa}\|^2 =
\Big|\prod_{j=\kappa-n}^{\kappa-1}\lambda_{-j}\Big|^2$
for $n \in J_\kappa$. Define the systems
$\{\mu_v\}_{v\in V_{\eta,\kappa}}$ and
$\{\varepsilon_v\}_{v\in V_{\eta,\kappa}}$ by
\eqref{kap0} and
   \begin{align*}
\mu_{-l}(\sigma) & = \frac{1}{\|\slam^{-l + \kappa}
e_{-\kappa}\|^2}\int_\sigma s^{-l + \kappa} \D\nu(s),
\quad \sigma \in \borel{\rbb_+},\, l\in J_\kappa \cup
\{0\},
   \\
\varepsilon_v & =
   \begin{cases}
0 & \text{ if } v\in V_{\eta,\kappa}^\circ,
   \\
\nu(\{0\}) & \text{ if } v = - \kappa.
   \end{cases}
   \end{align*}
Clearly $\mu_{-\kappa}=\nu$, which together with
\eqref{zgod0}, \eqref{prob} and \eqref{prob'} implies
that the systems $\{\mu_v\}_{v\in V_{\eta,\kappa}}$
and $\{\varepsilon_v\}_{v\in V_{\eta,\kappa}}$ satisfy
the assumptions of Theorem \ref{main}. As a
consequence, $\slam$ is subnormal.

(iv) Define the system $\{\mu_v\}_{v\in
V_{\eta,\kappa}}$ by \eqref{kap0}, \eqref{literki1}
and \eqref{literki2}. In view of (ii), the systems
$\{\mu_v\}_{v\in V_{\eta,\kappa}}$ and
$\{\varepsilon_v\}_{v\in V_{\eta,\kappa}}$ with
$\varepsilon_v\equiv 0$ satisfy the assumptions of
Theorem \ref{main}, and so $\slam$ is subnormal.
   \end{proof}
   It is worth mentioning that conditions (ii) and
(iii) of Theorem \ref{omega2} are equivalent without
assuming that \eqref{zgod0} is satisfied.
   \begin{lem} \label{IBJ}
Let $\slam$ be a weighted shift on the directed tree
$\tcal_{\eta,\kappa}$ with nonzero weights $\lambdab =
\{\lambda_v\}_{v \in V_{\eta,\kappa}^\circ}$ such that
$e_0 \in \dzn{\slam}$ and let $\{\mu_i\}_{i=1}^\eta$
be a sequence of Borel probability measures on
$\rbb_+$. Then conditions {\em (ii)} and {\em (iii)}
of Theorem {\em \ref{omega2}} $($with the same
$\kappa$$)$ are equivalent.
   \end{lem}
   \begin{proof}
(ii)$\Rightarrow$(iii) Let $\{\mu_{-l}\}_{l=0}^\kappa$
be the Borel probability measures on $\rbb_+$ defined
by \eqref{literki1}, \eqref{literki2} and
\eqref{literki3} with $\varepsilon_{-\kappa}$ given by
\eqref{literki4}. Set $\nu=\mu_{-\kappa}$. It follows
from \eqref{literki3} that for every $n \in J_\kappa$,
   \begin{align}  \label{intnu}
\int_\sigma s^n \D \nu(s) =
\Big|\prod_{j=0}^{\kappa-1} \lambda_{-j}\Big|^2 \;
\sum_{i=1}^\eta |\lambda_{i,1}|^2 \int_\sigma
\frac{1}{s^{\kappa + 1 - n}} \D \mu_i(s), \quad \sigma
\in \borel{\rbb_+}.
   \end{align}
This immediately implies \eqref{prob'}. By
\eqref{literki1}, \eqref{literki2} and \eqref{intnu},
we have
   \begin{align*}
\int_\sigma s^n \D \nu(s) =
   \begin{cases}
\cfrac{|\prod_{j=0}^{\kappa-1}
\lambda_{-j}|^2}{|\prod_{j=0}^{\kappa-n-1}
\lambda_{-j}|^2} \, \mu_{-(\kappa-n)}(\sigma) & \text{
if } n \in J_{\kappa-1},
   \\[3ex]
|\prod_{j=0}^{\kappa-1} \lambda_{-j}|^2 \,
\mu_{0}(\sigma) & \text{ if } n=\kappa,
   \end{cases}
   \end{align*}
for all $\sigma \in \borel{\rbb_+}$. Substituting
$\sigma=\rbb_+$ and using the fact that
$\{\mu_{-l}\}_{l=0}^{\kappa-1}$ are probability
measures, we obtain \eqref{prob}.

(iii)$\Rightarrow$(ii) Given $n \in J_\kappa$, we
define the positive Borel measure $\rho_n$ on $\rbb_+$
by $\rho_n(\sigma) = \int_\sigma s^n \D \nu(s)$ for
$\sigma \in \borel{\rbb_+}$. By \eqref{prob'},
equality \eqref{intnu} holds for $n=\kappa$. If this
equality holds for a fixed $n \in J_\kappa \setminus
\{1\}$, then $\rho_n(\{0\})=0$ and consequently
   \begin{align*}
\int_\sigma s^{n-1} \D \nu(s) = \int_\sigma
\frac{1}{s} \D \rho_n(s) \overset{\eqref{intnu}}=
\Big|\prod_{j=0}^{\kappa-1} \lambda_{-j}\Big|^2 \;
\sum_{i=1}^\eta |\lambda_{i,1}|^2 \int_\sigma
\frac{1}{s^{\kappa + 1 - (n-1)}} \D \mu_i(s)
   \end{align*}
for all $\sigma \in \borel{\rbb_+}$. Hence, by reverse
induction on $n$, \eqref{intnu} holds for all $n\in
J_\kappa$. Substituting $\sigma=\rbb_+$ into
\eqref{intnu} and using \eqref{prob}, we obtain
\eqref{zgod'} and \eqref{widly1}. It follows from
\eqref{intnu}, applied to $n=1$, that for every
$\sigma \in \borel{\rbb_+}$,
   \begin{multline}  \label{nusig}
\nu(\sigma) = \nu(\sigma \setminus \{0\}) + \nu(\{0\})
\delta_0(\sigma) = \int_\sigma \frac{1}{s} \D
\rho_1(s) + \nu(\{0\}) \delta_0(\sigma)
   \\
\overset{\eqref{intnu}}= \Big|\prod_{j=0}^{\kappa-1}
\lambda_{-j}\Big|^2 \; \sum_{i=1}^\eta
|\lambda_{i,1}|^2 \int_\sigma \frac{1}{s^{\kappa + 1}}
\D \mu_i(s) + \nu(\{0\}) \delta_0(\sigma).
   \end{multline}
Substituting $\sigma=\rbb_+$ into \eqref{nusig} and
using the fact that $\nu(\rbb_+)=1$, we obtain
\eqref{widly1'}. This completes the proof.
   \end{proof}
   Now we show that under some additional requirements
imposed on the weighted shift in question the
sufficient conditions appearing in Theorem
\ref{omega2} become necessary (see also Remark
\ref{deterrem} below).
   \begin{thm}  \label{deter}
   Let $\slam$ be a subnormal weighted shift on the
directed tree $\tcal_{\eta,\kappa}$ with nonzero
weights $\lambdab = \{\lambda_v\}_{v \in
V_{\eta,\kappa}^\circ}$. If $e_0 \in \dzn{\slam}$ and
   \begin{align}    \label{detn+1}
\text{$\Big\{\sum_{i=1}^\eta \Big|\prod_{j=1}^{n+1}
\lambda_{i,j}\Big|^2\Big\}_{n=0}^\infty$ is a determinate
Stieltjes moment sequence,}
   \end{align}
then the following four assertions hold\/{\em :}
   \begin{enumerate}
   \item[(i)] if $\kappa = 0$, then there exists a
sequence $\{\mu_i\}_{i=1}^\eta$ of Borel probability
measures on $\rbb_+$ that satisfy \eqref{zgod0} and
\eqref{zgod},
   \item[(ii)] if $0 < \kappa < \infty$, then there
exists a sequence $\{\mu_i\}_{i=1}^\eta$ of Borel
probability measures on $\rbb_+$ that satisfy
\eqref{zgod0}, \eqref{zgod'}, \eqref{widly1} and
\eqref{widly1'},
   \item[(iii)] if $0 < \kappa < \infty$, then there
exist a sequence $\{\mu_i\}_{i=1}^\eta$ of Borel
probability measures on $\rbb_+$ and a Borel
probability measure $\nu$ on $\rbb_+$ that satisfy
\eqref{zgod0}, \eqref{prob} and \eqref{prob'},
   \item[(iv)] if $\kappa=\infty$, then there
exists a sequence $\{\mu_i\}_{i=1}^\eta$ of Borel
probability measures on $\rbb_+$ that satisfy
\eqref{zgod0}, \eqref{zgod'} and \eqref{widly1}.
   \end{enumerate}
Moreover, if $e_0 \in \quasi{\slam}$, i.e.,
$\sum_{n=1}^\infty \big(\sum_{i=1}^\eta
\big|\prod_{j=1}^{n}
\lambda_{i,j}\big|^2\big)^{-\nicefrac{1}{2n}} =
\infty$, then \eqref{detn+1} is satisfied.
   \end{thm}
   \begin{proof}
It is clear that $e_0\in \dzn{\slam}$ implies that
$\mathscr{E}_{V_{\eta,\kappa}} \subseteq \dzn{\slam}$
and
   \begin{align} \label{detn+2}
\|\slam^{n+1} e_0\|^2 = \sum_{i=1}^\eta
\big|\prod_{j=1}^{n+1} \lambda_{i,j}\big|^2, \quad n
\in \zbb_+.
   \end{align}
By Proposition \ref{necess}, for every $u \in
V_{\eta,\kappa}$ the sequence $\{\|\slam^n
e_u\|^2\}_{n=0}^\infty$ is a Stieltjes moment
sequence. For each $i \in J_\eta$, we choose a
representing measure $\mu_i$ of $\{\|\slam^{n}
e_{i,1}\|^2\}_{n=0}^\infty$. It is easily seen that
\eqref{zgod0} holds. Since, by \eqref{detn+1} and
\eqref{detn+2}, the Stieltjes moment sequence
$\{\|\slam^{n+1} e_{0}\|^2\}_{n=0}^\infty$ is
determinate, we infer from Lemma \ref{charsub2},
applied to $u=0$, that \eqref{zgod} holds and
$\{\|\slam^{n} e_{0}\|^2\}_{n=0}^\infty$ is a
determinate Stieltjes moment sequence with the
representing measure $\mu_0$ given by
   \begin{align} \label{muu+2}
\mu_0(\sigma) = \sum_{i=1}^\eta |\lambda_{i,1}|^2
\int_\sigma \frac 1 s \D \mu_i(s) + \varepsilon_0
\delta_0(\sigma), \quad \sigma \in \borel{\rbb_+},
      \end{align}
where $\varepsilon_0$ is a nonnegative real number. In
view of the above, assertion (i) is proved.

Suppose $0 < \kappa \Le \infty$. Since $\{\|\slam^{n}
e_{0}\|^2\}_{n=0}^\infty$ is a determinate Stieltjes
moment sequence, we deduce from Lemma \ref{charsub-1},
applied to $u_0=-1$, that $\{\|\slam^{n+1}
e_{-1}\|^2\}_{n=0}^\infty$ and $\{\|\slam^{n}
e_{-1}\|^2\}_{n=0}^\infty$ are determinate Stieltjes
moment sequences~ and
   \begin{align} \label{jabko}
& \int_0^\infty \frac 1 s \D \mu_0(s) \Le
\frac{1}{|\lambda_{0}|^2},
   \\
& \mu_{-1}(\sigma) = |\lambda_{0}|^2 \int_\sigma \frac
1 s \D \mu_0(s) + \varepsilon_{-1} \delta_0(\sigma),
\quad \sigma \in \borel{\rbb_+}, \label{jabko2}
   \end{align}
where $\mu_{-1}$ is the representing measure of
$\{\|\slam^{n} e_{-1}\|^2\}_{n=0}^\infty$ and
$\varepsilon_{-1}$ is a nonnegative real number.
Inequality \eqref{jabko} combined with equality
\eqref{muu+2} implies that $\varepsilon_0=0$ and
therefore that \eqref{widly1'} holds for $\kappa=1$.
Substituting $\sigma=\rbb_+$ into \eqref{muu+2}, we
obtain \eqref{zgod'}. This completes the proof of
assertion (ii) for $\kappa=1$. Note also that
equalities \eqref{muu+2} and \eqref{jabko2}, combined
with $\varepsilon_0=0$, yield
   \begin{align*}
\mu_{-1}(\sigma) = |\lambda_{0}|^2 \sum_{i=1}^\eta
|\lambda_{i,1}|^2 \int_\sigma \frac 1 {s^2} \D
\mu_i(s) + \varepsilon_{-1} \delta_0(\sigma), \quad
\sigma \in \borel{\rbb_+}.
   \end{align*}
If $\kappa > 1$, then arguing by induction, we
conclude that for every $k\in J_{\kappa}$ the
Stieltjes moment sequences $\{\|\slam^{n+1}
e_{-k}\|^2\}_{n=0}^\infty$ and $\{\|\slam^{n}
e_{-k}\|^2\}_{n=0}^\infty$ are determinate and
   \begin{align} \label{mu-l}
\mu_{-l}(\sigma) =
\Big|\prod_{j=0}^{l-1}\lambda_{-j}\Big|^2
\sum_{i=1}^\eta |\lambda_{i,1}|^2 \int_\sigma \frac 1
{s^{l+1}} \D \mu_i(s), \quad \sigma \in
\borel{\rbb_+}, \, l\in J_{\kappa - 1},
   \end{align}
where $\mu_{-l}$ is the representing measure of
$\{\|\slam^{n} e_{-l}\|^2\}_{n=0}^\infty$.
Substituting $\sigma=\rbb_+$ into \eqref{mu-l}, we
obtain \eqref{widly1}. This completes the proof of
assertion (iv). Finally, if $1 < \kappa < \infty$,
then again by Lemma \ref{charsub-1}, now applied to
$u=-\kappa$, we have $\int_0^\infty \frac 1 s \D
\mu_{-\kappa+1}(s) \Le
\frac{1}{|\lambda_{-\kappa+1}|^2}$. This inequality
together with \eqref{mu-l} yields \eqref{widly1'},
which completes the proof of assertion (ii).

Assertion (iii) can be deduced from assertion (ii) via
Lemma \ref{IBJ}.

Arguing as in the proof of Theorem \ref{main-0}, we
see that if $e_0 \in \quasi{\slam}$, then
\eqref{detn+1} is satisfied.
   \end{proof}
   \begin{rem}  \label{deterrem}
A careful look at the proof reveals that Theorem
\ref{deter} remains valid if instead of assuming that
$\slam$ is subnormal, we assume that $\{\|\slam^n
e_u\|^2\}_{n=0}^\infty$ is a Stieltjes moment sequence
for every $u \in \big\{-k\colon k\in J_\kappa\big\}
\sqcup \{0\} \sqcup \dzi{0}$.
   \end{rem}
   \begin{cor}
   Let $\slam$ be a weighted shift on the directed
tree $\tcal_{\eta,\kappa}$ with nonzero weights
$\lambdab = \{\lambda_v\}_{v \in
V_{\eta,\kappa}^\circ}$ such that $e_0 \in
\dzn{\slam}$ $($or, equivalently,
$\mathscr{E}_{V_{\eta,\kappa}} \subseteq
\dzn{\slam}$$)$. Suppose that $\{\|\slam^n
e_v\|^2\}_{n=0}^\infty$ is a Stieltjes moment sequence
for every $v \in \big\{-k\colon k\in J_\kappa\big\}
\sqcup \{0\} \sqcup \dzi{0}$, and that
$\{\|\slam^{n+1} e_0\|^2\}_{n=0}^\infty$ is a
determinate Stieltjes moment sequence. Then the
following assertions hold\/{\em :}
   \begin{enumerate}
   \item[(i)] $\slam$ is subnormal,
   \item[(ii)] $\{\|\slam^{n+1}
e_{-j}\|^2\}_{n=0}^\infty$ is a determinate Stieltjes
moment sequence for every integer $j$ such that $0 \Le
j \Le \kappa$,
   \item[(iii)] $\slam$  satisfies the consistency condition
\eqref{alanconsi} at the vertex $u=-j$ for every
integer $j$ such that $0 \Le j \Le \kappa$.
   \end{enumerate}
   \end{cor}
   \begin{proof}
   (i) By using Remark \ref{deterrem}, we find a
sequence $\{\mu_i\}_{i=1}^\eta$ of Borel probability
measures on $\rbb_+$ satisfying \eqref{zgod0} and
exactly one of the conditions (i), (ii) and (iv) of
Theorem \ref{omega2} (the choice depends on $\kappa$),
and then we apply Theorem \ref{omega2}.

   (ii) See the proof of Theorem \ref{deter}.

   (iii) Apply (ii) and Lemma \ref{charsub2}\,(ii).
   \end{proof}
   \subsection*{Acknowledgements} A substantial part of
this paper was written while the second and the fourth
authors visited Kyungpook National University during
the autumn of 2010 and the spring of 2011. They wish
to thank the faculty and the administration of this
unit for their warm hospitality.
   \bibliographystyle{amsalpha}
   
   \end{document}